\documentclass{elsarticle}

\usepackage[boxed,linesnumbered]{algorithm2e}
\usepackage{amsmath}
\usepackage{amssymb}
\usepackage{amsthm}
\usepackage{latexsym,pifont}
\usepackage{epsfig}
\newtheorem{clm}{Claim}[section]
\newtheorem{remark}{Remark}
\usepackage{pslatex}

\numberwithin{equation}{section}

\def\Av{{\bf A}}
\def\B{{\bf B}}

\def\u{{\bf u}}

\def\x{{\bf x}}
\def\n{{\bf n}}

\def\En{{\mathcal E}}
\def\half{\frac{1}{2}}

\def\divg{\nabla\cdot}
\def\curl{\nabla \times}
\def\Dx{\Delta x}
\def\Dphi{\Delta^+ q}

\def\be{\begin{equation}}
\def\ee{\end{equation}}

\def\ba{\begin{align}}
\def\ea{\end{align}}

\journal{Journal of Computational Physics}

\begin{document}

\begin{frontmatter}



\title{Finite Difference Weighted Essentially Non-Oscillatory Schemes
	with Constrained Transport for Ideal Magnetohydrodynamics}

\author[author1]{Andrew J. Christlieb}
\ead{christli@msu.edu}


\author[author2]{James A. Rossmanith\fnref{labc}}
\ead{rossmani@iastate.edu}

\author[author3]{Qi Tang}
\ead{tangqi@msu.edu}

\address[author1]{Department of Mathematics and Department of Electrical and Computer Engineering,
Michigan State University, East Lansing, MI 48824, USA}

\address[author2]{Department of Mathematics, Iowa State University, 
396 Carver Hall, Ames, IA 50011, USA}

\address[author3]{Department of Mathematics, Michigan State University, East Lansing, MI 48824, USA}

\fntext[labc]{Corresponding author}

\begin{abstract}

In this work we develop a class of high-order finite difference weighted essentially non-oscillatory (FD-WENO) 
schemes  for solving the ideal magnetohydrodynamic (MHD) equations in 2D and 3D. 
The philosophy of this work is to use efficient high-order WENO spatial discretizations with
 high-order strong stability-preserving Runge-Kutta (SSP-RK) time-stepping schemes.
 Numerical results have shown that with such methods we are able to resolve
 solution structures that are only visible at much higher grid resolutions with lower-order schemes.
The key challenge in applying such methods to ideal MHD is to control divergence errors in the magnetic field.
We achieve this by augmenting the base scheme with  a novel high-order constrained transport
approach that updates the magnetic vector potential. The predicted magnetic field from the base scheme
is replaced by a divergence-free magnetic field that is obtained from the curl of this magnetic potential.
The non-conservative weakly hyperbolic
system that the magnetic vector potential satisfies is solved using a version of FD-WENO
developed for Hamilton-Jacobi equations. 
The resulting numerical method is endowed with  several important properties: (1) all quantities, including all components of the magnetic field and magnetic potential, are treated as point values on the same mesh (i.e., there is no mesh staggering); 
(2) both the spatial and temporal orders of accuracy are  fourth-order; (3) no spatial integration or multidimensional reconstructions are needed in any step; and (4) special limiters in the magnetic vector potential update are used to control unphysical oscillations
in the magnetic field. 
Several 2D and 3D numerical examples are presented to verify the order of accuracy on smooth test problems and to show high-resolution on test problems that involve shocks.

\end{abstract}

\end{frontmatter}

\section{Introduction}
The ideal magnetohydrodynamics (MHD) equations model the dynamics of a perfectly conducting
quasi-neutral plasma, and provide evolution equations for the macroscopic quantities of mass, momentum, and
energy density, as well as the magnetic field.  
MHD has been used successfully as a model in several application areas, including in space weather prediction, 
astrophysics, as well as in laboratory plasma applications such as flows in tokamaks and stellarators.
Mathematically, the MHD equations are a set of nonlinear hyperbolic conservation laws with the additional restriction that the magnetic field must remain
divergence-free for all time. In fact, at the continuum level, if the initial magnetic field is divergence-free, then the ideal MHD equations propagate this condition forward for all time.
Unfortunately, most standard numerical discretizations based on shock-capturing methods (e.g., finite volume,
weighted essentially non-oscillatory, discontinuous Galerkin) do not propagate a discrete version of the divergence-free
condition forward in time; and furthermore, this failure has been shown in the literature to
produce numerical instabilities.
The main  challenge in numerically simulating the ideal MHD system is therefore to augment existing schemes
so that they satisfy a divergence-free condition of the magnetic field in some way. 
Roughly speaking, there are four kinds of approaches that have been proposed in the literature:
(1) the  8-wave formulation \cite{article:Po94,article:Po99}, 
(2) projection methods \cite{article:BaKi04,article:To00}, 
(3) hyperbolic divergence-cleaning methods \cite{article:DKKMSW02},
and (4) constrained transport methods \cite{article:Ba04,article:BaSp99a,article:DaWo98,article:EvHa88,article:FeTo03,article:HeRoTa10,article:LoZa00,article:LoZa04,article:Ro04b,article:Ryu98,article:De01b,article:To00,article:To2005}.

In the current work we focus on the  {\it constrained transport} (CT) methodology for producing a magnetic field
that satisfies a discrete divergence-free condition. The CT method was originally introduced by Evans and Hawley 
\cite{article:EvHa88}, and, in their formulation, staggered electric and magnetic fields are used to create
appropriate mimetic finite difference operators that ultimately lead to an exactly divergence-free magnetic field. 
Their constrained transport framework is a modification of the popular Yee scheme \cite{article:Ye66}
from electromagnetics to ideal MHD. 

Since the introduction of the constrained transport methodology, there have been many modifications and extensions, 
especially in the context of high-resolution shock-capturing schemes. DeVore \cite{article:DeVore91} developed a 
flux-corrected transport implementation of the constrained transport approach. Balsara and Spicer \cite{article:BaSp99a}, Dai and Woodward \cite{article:DaWo98}, and Ryu et al.\ \cite{article:Ryu98},  all developed
various strategies for constructing the electric field via Ohm's law in the constrained transport framework.
Londrillo and Zanna \cite{article:LoZa00, article:LoZa04} proposed a high-order version of the
Evans and Hawley approach. De Sterck \cite{article:De01b} developed a similar CT method on unstructured triangular grids. 
Balsara \cite{article:Ba01c} developed from the constrained transport framework an adaptive mesh refinement (AMR) 
scheme that also included a globally divergence-free magnetic field reconstruction. 
There is a careful description and comparison of several of these methods in the 
article of T\'{o}th \cite{article:To00}, in which he also showed that a staggered magnetic field is not necessary,
and then introduced several unstaggered constrained transport methods. 

In recent years, unstaggered CT methods have attracted considerable interest due to their ease of implementation and applicability to adaptive mesh refinement (AMR) strategies. For instance, Fey and Torrilhon \cite{article:FeTo03} presented a way to preserve divergence-free condition through an unstaggered upwind scheme. Rossmanith \cite{article:Ro04b} developed an unstaggered CT method for the 2D MHD equations on Cartesian grids based on the wave-propagation method \cite{article:Le97}. Helzel et al.\ \cite{article:HeRoTa10,article:helzel2013} 
extended this unstaggered CT method to the 3D MHD equations and to mapped grids. 

In addition to the above mentioned papers, several high-order methods have been proposed in recent years
for the ideal MHD equations using a variety of discretization techniques. 
Balsara \cite{article:Bal09} developed a third-order RK-WENO method
using a staggered magnetic field to reconstruct a globally divergence-free magnetic field. 
Balsara et al. \cite{article:Bal13-ader,article:Bal09-ader} developed a class of high-order ADER-WENO schemes,
again using a staggered magnetic field to reconstruct a globally divergence-free magnetic field. 
Li et al. \cite{article:LiXuYa11} and Cheng et al. \cite{article:Cheng13} 
introduced a class of central discontinuous Galerkin schemes that evolves the MHD equations on
a primal as well as a dual mesh, and by intertwining these two updates, they showed
that a globally divergence-free magnetic field could be obtained.
Kawai \cite{article:Kawai13} devloped a high-order finite difference method with artificial resistivity,
where the finite difference operators were specifically constructed to guarantee that an appropriate
definition of the magnetic field divergence is propagated forward in time by the
numerical scheme.

The focus of the current work is to develop a class of high-order finite difference weighted essentially non-oscillatory (FD-WENO)  schemes  for solving the ideal magnetohydrodynamic (MHD) equations in 2D and 3D.  We make use of the
basic unstaggered constrained transport framework developed by Helzel et al. \cite{article:helzel2013}, although we
use different temporal and spatial discretizations, 
and we develop a different numerical technique for updating the magnetic vector potential and for correcting
the magnetic field. All aspects of the proposed numerical scheme are higher-order accurate (for smooth solutions)
than the method that is developed in \cite{article:helzel2013}.
A summary of the key features of the proposed  numerical method are listed below:
\begin{enumerate}
\item All quantities, including all components of the magnetic field and magnetic potential, are treated as point values on the same mesh (i.e., there is no mesh staggering).
\item The base scheme is the $5^{\text{th}}$-order FD-WENO scheme of Jiang and Shu \cite{article:JiShu96}. 
With this method we are able to achieve high-order using dimension-by-dimension finite difference operators, instead
of the more complicated spatial integration and multidimensional reconstructions used by Helzel et al. \cite{article:helzel2013}.
\item The corrected magnetic field is computed via $4^{\text{th}}$-order accurate central finite
difference operators  that approximate the curl of the magnetic vector potential. 
These operators are chosen to produce a corrected magnetic field that exactly satisfies a discrete divergence-free condition.
\item All time-stepping is done with the $4^{\text{th}}$-order accurate, 10-stage, low-storage, strong stability preserving Runge-Kutta (SSP-RK) method developed by Ketcheson \cite{article:Ketch08}.
\item Using a particular gauge condition, the magnetic vector potential is made to satisfy a weakly hyperbolic, non-conservative, hyperbolic system. This system is solved using a modified version of the FD-WENO scheme developed for 
Hamilton--Jacobi equations \cite{article:jiang2000}. Special limiters based on  artificial resistivity are introduced to 
help control unphysical oscillations in the magnetic field. 
\end{enumerate}
The numerical methods presented in this work are a first step in a larger effort to develop
a multi-purpose high-order FD-WENO code with adaptive mesh refinement (AMR) based on the
work Shen et al. \cite{article:ShQiCh}. Ultimately, we envision that this AMR code 
will have the capability to solve ideal MHD, as well as several other models from plasma physics.

The outline of the paper is as follows. In Section \ref{sec:mhd_eqns}  we will briefly review the MHD equations and the evolution
equations for the magnetic vector potential. In Sections \ref{ctmtd}--\ref{sdcurl} we describe different parts of the numerical
discretization of the proposed FD-WENO constrained transport method. In Section \ref{ctmtd} we outline the time-stepping
procedure in our unstaggered constrained transport methodology, including the predictor and corrector
steps in the update of the MHD variables. In Section \ref{sdmhd} we detail the $5^{\text{th}}$-order 
FD-WENO spatial discretization for the MHD system.
In Section \ref{sdmpe} we describe the spatial discretization of the 2D scalar potential evolution equations and the 3D vector potential equations. For the 3D case, we also develop an artificial resistivity term for controlling unphysical oscillations
in the magnetic field. The numerical curl operator and its properties are then discussed in Section 6, which completes all the steps of our CT methods. The resulting 2D and 3D schemes are implemented and tested on several numerical examples in Section \ref{numer}. 

\section{The ideal MHD equations}
\label{sec:mhd_eqns}

The ideal MHD equations in  conservation form can be written as
\begin{gather}
\label{MHD}
  \frac{\partial}{\partial t}
   \begin{bmatrix}
     \rho \\ \rho \u \\ \En\\ \B
    \end{bmatrix}
    + \divg{
    	\begin{bmatrix}
	  \rho \u \\ 
	  \rho \u \otimes \u + (p+\frac{1}{2} \| \B \|^2) \mathbb{I} - \B \otimes \B \\
	  \u (\mathcal{E} +  p +\frac{1}{2} \| \B \|^2 ) - \B(\u \cdot \B) \\
	  \u \otimes \B- \B \otimes \u 
	\end{bmatrix}
	}
	 = 0, \\
\label{constraint}
\divg{ \B} = 0,
\end{gather} 
where $\rho$, $\rho \u$, and $\En$ are the total mass, momentum and energy densities of the system, $\B$ is the magnetic field,  and $p$ is the hydrodynamic pressure. The total energy density is given by
\begin{align}
\label{eq:eng}
	\mathcal{E} =\frac{p}{\gamma-1} + \half {\rho \|\u\|^2}   + \half {\|\B\|^2},
\end{align}
where $\gamma = 5/3$ is the ideal gas constant. Here $\| \cdot \|$ is used to denote the Euclidean vector norm. A complete derivation of MHD system \eqref{MHD}--\eqref{constraint} can be found in many standard plasma physics textbooks (e.g., pages 165--190 of \cite{book:Pa04}).

\subsection{Hyperbolicity of the governing equations}
\label{hypermhd}
Equations $\eqref{MHD}$, along with the equation of state $\eqref{eq:eng}$, form
a system of  hyperbolic conservation laws:
\begin{equation}
	q_{,t} + \nabla \cdot {\bf F}(q) = 0,
\end{equation}
where   $q = (\rho,\rho \u, \En, \B)$ are the conserved variables and ${\bf F}$ is the flux tensor (see \eqref{MHD}).
Under the assumption of positive pressure ($p>0$) and density ($\rho>0$), the flux Jacobian in some arbitrary direction $\n$ ($\| \n \| = 1$), $A(q; {\bf n}) := \n \cdot {\bf F}_{,q}$, is a diagonalizable
matrix with real eigenvalues. 
In particular, the eigenvalues of the flux Jacobian matrix in some arbitrary direction $\n$ ($\|\n\|=1$) can be written as follows:
\begin{alignat}{2}
  \lambda^{1,8} &= {\bf u} \cdot {\bf n} \mp c_f & & \qquad \text{: fast magnetosonic waves,} \\
  \lambda^{2,7} &= {\bf u} \cdot {\bf n}  \mp c_a & & \qquad \text{: Alfv$\acute{\text{e}}$n waves,} \\
  \lambda^{3,6} &= {\bf u} \cdot {\bf n}  \mp c_s & & \qquad \text{: slow magnetosonic waves,} \\
  \lambda^{4} &= {\bf u} \cdot {\bf n} & & \qquad \text{: entropy wave,} \\
  \lambda^{5} &= {\bf u} \cdot {\bf n}  & & \qquad \text{: divergence wave,}
\end{alignat}
where
\begin{align}
a & \equiv \sqrt{\frac{\gamma p}{\rho}}, \\
c_a & \equiv \sqrt{\frac{(\B \cdot \n)^2}{\rho}}, \\
c_f & \equiv \left\{ \frac{1}{2} \left[ a^2+\frac{\|\B\|^2}{\rho} + \sqrt{ \left(a^2+\frac{\|\B\|^2}{\rho}\right)^2 - 4a^2 \frac{(\B \cdot \n)^2}{\rho} } \right] \right\}^{\frac{1}{2}},\\ 
c_s & \equiv \left\{ \frac{1}{2} \left[ a^2+\frac{\|\B\|^2}{\rho} - \sqrt{ \left(a^2+\frac{\|\B\|^2}{\rho}\right)^2 - 4a^2 \frac{(\B \cdot \n)^2}{\rho} } \right] \right\}^{\frac{1}{2}}.
\end{align}
The eight eigenvalues are well-ordered in the sense that
\begin{align}
\lambda^1 \le \lambda^2 \le \lambda^3 \le \lambda^4 \le \lambda^5 \le \lambda^6 \le \lambda^7 \le \lambda^8.
\end{align}


\subsection{Magnetic potential in 3D}
\label{mhd-eq-mq3d}
Although there are many available numerical methods for solving hyperbolic systems
(e.g., finite volume, WENO, and discontinuous Galerkin) and most of them can in principle be directly used to simulate the MHD systems, the main challenge in numerically solving the MHD equations is related to the divergence-free condition on the magnetic field. First, we note that the MHD system \eqref{MHD} along with \eqref{eq:eng} is already a closed set of eight evolution equations.
Second, we note that $\divg \B = 0$ is an {\it involution} instead of a constraint (see page 119--128 of \cite{book:Dafermos10}),
because if $\divg \B = 0$ is satisfied initially ($t = 0$), then system \eqref{MHD} guarantees that $\divg \B = 0$ is satisfied for all future time ($t>0$). 
Unfortunately, most numerical discretizations of MHD do not propagate some discrete version of $\divg \B = 0$ forward
in time. As has been shown repeatedly in the literature, failure to adequately control the resulting divergence errors can lead to numerical instability (see e.g., \cite{article:BaKi04,article:LiShu05,article:Po94,article:Ro04b,article:De01b,article:To00}).
To address this issue, we will make use of the magnetic potential in the numerical methods described in this work. 

Because it is divergence-free, the magnetic field can be written as the curl of a magnetic vector potential:
\begin{align}
\label{ptform}
\B = \nabla \times \Av.
\end{align}
Furthermore, we can write the evolution equation of the magnetic field in the MHD systems $\eqref{MHD}$
in curl form:
\begin{align}
\label{bevolve}
\B_{,t} + \nabla \times (\B \times \u) = 0,
\end{align}
due to the following relation
\begin{align}
\divg \left( \u \otimes \B - \B \otimes \u\right) = \curl(\B \times \u).
\end{align}
Using the magnetic vector potential $\eqref{ptform}$,   evolution equation $\eqref{bevolve}$ can be written as
\begin{align}
\label{curlA}
\curl \left\{ \Av_{,t} + (\curl \Av)  \times \u \right\} = 0.
\end{align}
The relation $\eqref{curlA}$ implies the existence of a scalar function $\psi$ such that
\begin{align}
\Av_{,t} + (\curl \Av)  \times \u = -\nabla \psi.
\end{align}
In order to uniquely (at least up to additive constants) determine the additional scalar function $\psi$, 
we must prescribe some gauge condition.

After investigating several gauge conditions, Helzel et al.  \cite{article:HeRoTa10} found that 
one can obtain stable solutions by introducing the Weyl gauge, i.e., setting $\psi \equiv 0$. With this gauge choice, the evolution equation for the vector potential becomes
\begin{align}
\Av_{,t} + (\curl \Av)  \times \u = 0,
\end{align}
which can be rewritten as a non-conservative quasilinear system,
\begin{align}
\label{3dpte}
\Av_{,t} + N_1 \, \Av_{,x} + N_2 \, \Av_{,y} + N_3 \, \Av_{,z} = 0,
\end{align}
where
\begin{align}
\hspace{-3mm}
\label{eqn:nmats}
N_1 =  \begin{bmatrix}
  0 & -u^2 & -u^3 \\
  0 & u^1 & 0 \\
  0 & 0 & u^1
 \end{bmatrix}, 
 N_2 =   \begin{bmatrix}
  u^2 & 0 & 0 \\
  -u^1 & 0 & -u^3 \\
  0 & 0 & u^2
 \end{bmatrix}, 
 N_3 =  \begin{bmatrix}
  u^3 & 0 & 0 \\
  0 & u^3 & 0 \\
  -u^1 & -u^2 & 0
 \end{bmatrix}.
\end{align}
One difficulty with system \eqref{3dpte}--\eqref{eqn:nmats} is that it is only weakly hyperbolic \cite{article:HeRoTa10}. 
In order to see this weak hyperbolicity, we start with the flux Jacobian matrix in some arbitrary direction $\n = (n^1,n^2,n^2)$:
\begin{align}
\label{whmatrix}
n^1 N_1 + n^2 N_2 + n^3 N_3 = 
\begin{bmatrix}
  n^2u^2 + n^3u^3 & -n^1u^2 & -n^1u^3 \\
  -n^2u^1  & n^1u^1 + n^3u^3 & -n^2u^3 \\
  -n^3u^1  & -n^3u^2 & n^1u^1 + n^2u^2
 \end{bmatrix}.
\end{align}
The eigenvalues of matrix $\eqref{whmatrix}$ are
\begin{align}
\lambda^1 = 0, \quad \lambda^{2} = \lambda^{3} = \n\cdot\u,
\end{align}
and the matrix of right eigenvectors can be written as
\begin{align}
\label{whmatrix2}
R =  \left[ \begin{array}{c|c|c}
   &  &  \\
  r^{(1)} & r^{(2)} & r^{(3)} \\
    &  & 
 \end{array} \right] = 
\begin{bmatrix}
  n^1 & n^2u^3-n^3u^2 & u^1(\u\cdot\n)-n^1\|\u\|^2 \\
  n^2 & n^3u^1-n^1u^3 & u^2(\u\cdot\n)-n^2\|\u\|^2 \\
  n^3 & n^1u^2-n^2u^1 & u^3(\u\cdot\n)-n^3\|\u\|^2
 \end{bmatrix}.
\end{align}
If we assume that $\|\u\| \neq 0$ and $\|\n\|=1$, the determinant of matrix $R$ is 
\begin{align}
\det(R) = - \|\u\|^3\cos(\alpha)\sin(\alpha),
\end{align}
where $\alpha$ is the angle between $\n$ and $\u$. In particular, there exist four degenerate directions, $\alpha =0, \pi/2, \pi, \text{and } 3\pi/2$, in which the eigenvectors are incomplete. Hence, the system $\eqref{3dpte}$ is only weakly hyperbolic. 

\subsection{Magnetic potential in 2D}
A special case of the situation described above is the MHD system in 2D. In particular, what we mean
by 2D is that all eight conserved variables, $q = (\rho,\rho \u, \En, \B)$, can be non-zero, but each depends
on only three independent variables: $t$, $x$, and $y$.
From the point-of-view of the magnetic potential, the 2D case
is much simpler than the full 3D case, due to the fact that the divergence-free condition simplifies to
\begin{align}
\label{2dcons}
\nabla \cdot \B = B^1_{, x} + B^2_{,y} = 0.
\end{align}
It can be readily seen that solving $B^3$ by any numerical scheme will not have any impact on the satisfaction of the divergence-free condition $\eqref{2dcons}$. In other words, using the magnetic potential to define $B^3$ is unnecessary. 
Instead, in the 2D case, we can write
\begin{align}
\label {2dcurl}
B^1 =  A^3_{, y} \quad \text{and} \quad B^2 = - A^3_{, x},
\end{align}
which involves only the third component of the magnetic potential, thereby effectively reducing the magnetic
{\it vector} potential to a {\it scalar} potential.
Consequently, the constrained transport method in 2D can be simplified to solving an advection equation for the third component of the vector potential:
\begin{align}
\label{2dpte}
 A^3_{, t} + u^1  A^3_{,x} + u^2 A^3_{,y} = 0.
\end{align}
This has the added benefit that \eqref{2dpte} is {\it strongly} hyperbolic, unlike its counterpart in the 3D case.

\section{Time discretization in the constrained transport framework}
\label{ctmtd}

In the unstaggered constrained transport (CT)
method for the ideal MHD equations developed by Helzel et al. \cite{article:helzel2013},
a conservative finite volume hyperbolic solver for the MHD equations is coupled to a non-conservative finite volume solver for the vector potential equation. Momentarily leaving out the details of our spatial discretization, the proposed method in this work
can be put into the same basic framework as that of Helzel et al. \cite{article:helzel2013}. 
We summarize this framework below.

We write the semi-discrete form of MHD equations $\eqref{MHD}$ as follows:
\begin{align}
\label{smmhd}
Q'_{\text{MHD}}(t) = \mathcal{L}(Q_{\text{MHD}}(t)),
\end{align}
where $Q_{\text{MHD}} = \left( \rho, \rho \u, {\mathcal E}, \B \right)$.
We write the semi-discrete form of the evolution equation of the magnetic potential (\eqref{3dpte} in 3D and \eqref{2dpte} in
 2D) as follows:
\begin{align}
\label{smmpt}
Q'_{\text{A}}(t) = \mathcal{H}(Q_{\text{A}}(t),\u(t)),
\end{align}
where $Q_{\text{A}} = (A^1, A^2, A^3)$ in 3D and $Q_{\text{A}} = A^3$ in 2D.

For simplicity, we present the scheme using only forward Euler time-stepping. 
Extension to high-order strong stability-preserving Runge-Kutta (SSP-RK) methods is straightforward,
since SSP-RK time-stepping methods are  convex combinations of forward Euler steps. A 
single time-step of the CT method from time $t=t^n$ to time $t=t^{n+1}$ consists of the following sub-steps:

\begin{enumerate}
\setcounter{enumi}{-1}
\item Start with $Q^n_{\text{MHD}}$ and $Q^n_\text{A}$ (the solution at $t^n$).

\item Build the right-hand sides of both semi-discrete systems $\eqref{smmhd}$ and $\eqref{smmpt}$ using 
the FD-WENO spatial discretizations that will be described in detail in \S \ref{sdmhd} and \S \ref{sdmpe}, and independently update each system:
\begin{align}
Q_{\text{MHD}}^{*} & =  Q_{\text{MHD}}^n + \Delta t \, \mathcal{L} (Q_{\text{MHD}}^{n}), \\
\label{eqn:potential-update}
Q_{\text{A}}^{n+1} & = Q_{\text{A}}^n + \Delta t \, \mathcal{H} (Q_{\text{A}}^{n}, \u^{n}),
\end{align}
where $Q_{\text{MHD}}^{*} = (\rho^{n+1}, \rho \u^{n+1}, \En^{*}, \B^{*})$. $\B^{*}$ is the predicted magnetic field that
in general does not satisfy a discrete divergence-free constraint and $\En^{*}$ is the predicted energy.
\item Replace $\B^{*}$ by a discrete curl of the magnetic potential $Q_{\text{A}}^{n+1}$:
\begin{align}
\B^{n+1} = \curl Q_{\text{A}}^{n+1}.
\end{align}
This discrete curl will be defined appropriately so that $\nabla \cdot \B^{n+1} = 0$ for some appropriate
definition of the discrete divergence operator (see \S \ref{sdcurl} for details).
\item Set the corrected total energy density $\En^{n+1}$ based on the following options:
	{\begin{enumerate}
	\item[] \textbf{Option 1:} Conserve the total energy: 
								\begin{align}
								\En^{n+1} = \En^{*}.
								\end{align}
	\item[] \textbf{Option 2:} Keep the pressure the same before and after the magnetic field 
		correction step ($p^{n+1} = p^{*}$): 
								\begin{align}
								\En^{n+1} = \En^{*}+\frac{1}{2} \left( \|\B^{n+1}\|^2 -\|\B^{*}\|^2 \right).
								\end{align}
This option may help to preserve the positivity of the pressure in certain problems, and thus can lead
to improved numerical stability, albeit at the expense of sacrificing energy conservation.
	\end{enumerate}}
\end{enumerate}

In this work we make exclusive use of \textbf{Option 1} and thus conserve the total energy.
Furthermore, recent work on positivity limiters such as the work of Zhang and Shu 
\cite{article:ZhangShu11} (for compressible Euler)  and Balsara \cite{article:Balsara12} (for MHD), 
have shown that it is possible to achieve pressure positivity even when conserving the total energy. 
Application of such positivity limiters to the currently proposed scheme
 is something that we hope to investigate in future work.

For the spatial discretizations of the MHD operator in \eqref{smmhd} and the
 the magnetic potential operator in \eqref{smmpt} we use $5^{\text{th}}$-order accurate
 finite difference WENO spatial discretizations (see \S \ref{sdmhd} and \S \ref{sdmpe}). For the discrete curl
 operator we use a $4^{\text{th}}$-order accurate central finite difference operator (see \S \ref{sdcurl}).
 In order to also achieve high-order accuracy in time, we make use of the 
 10-stage 4th-order strong stability-preserving Runge--Kutta (SSP-RK) scheme described in Ketcheson \cite{article:Ketch08}.
Although this method has 10-stages, it has a simple low-storage implementation that  requires the storage
of only two solution vectors. This method is detailed in Algorithm \ref{alg:ssp-rk}.
 \begin{algorithm}
  \SetAlgoNoLine
  $Q^{(1)} = Q^{n}$\; $Q^{(2)} = Q^{n}$\;
  \For{$i = {1:5}$}{
	$Q^{(1)} = Q^{(1)} + \frac{\Delta t}{6} \, \mathcal{L} \left(Q^{(1)} \right)$\;
  }
  $Q^{(2)} =  \frac{1}{25} \, Q^{(2)} + \frac{9}{25}\,Q^{(1)}$\;
  $Q^{(1)} = 15 \, Q^{(2)} - 5 \, Q^{(1)}$\;
  \For{$i = {6:9}$}{
	$Q^{(1)} = Q^{(1)} + \frac{\Delta t}{6}\, \mathcal{L} \left(Q^{(1)} \right)$\;
  }
  $Q^{n+1} =  Q^{(2)} + \frac{3}{5}\,Q^{(1)} + \frac{\Delta t}{10} \, \mathcal{L} \left(Q^{(1)} \right)$\;
  \caption{Low-storage SSP-RK4 method of Ketcheson \cite{article:Ketch08}. \label{alg:ssp-rk}}
\end{algorithm}

Since the SSP-RK4 method is a convex combination of forward Euler steps, using it in conjunction
with the above constrained transport framework turns out to be straightforward. In the
 spatial discretization of $\eqref{smmpt}$, the velocity $\u$ is always taken as a given function, and hence 
  equation $\eqref{smmpt}$ is treated as a closed equation for the magnetic potential. In addition, the correction of $\B^*$  (Step 2) is performed in each stage of an SSP-RK4 time-step. For smooth problems, this overall procedure gives a solution that is $4^{\text{th}}$-order accurate. This will be confirmed via numerical convergence studies in Section~\ref{numer}. 

Gottlieb et al. \cite{article:GoKeSh09} pointed out that the SSP-RK4 time-stepping scheme 
coupled to the $5^{\text{th}}$-order FD-WENO method 
results in essentially non-oscillatory solutions for CFL numbers up to $3.07$.
We confirmed this results by applying this method on the 1D Brio-Wu shock-tube problem \cite{article:BrWu88}.
In Section~\ref{numer} we use a CFL number of $3.0$ for all the 2D and 3D numerical examples and obtain good results. The fact that we are able to use a larger CFL number offsets
the computational cost of having to compute 10-stages for every time-step. 
As a point of comparison, we recall the 4-stage $4^{\text{th}}$-order Runge--Kutta method (RK44) used
by Jiang and Wu \cite{article:JiWu99} in their high-order FD-WENO scheme for ideal MHD.
First, we note that the RK44 scheme used in \cite{article:JiWu99} is not strong stability-preserving (SSP).
Second, a typical CFL number of RK44 with $5^\text{th}$-order FD-WENO is $0.8$. Consequently, even though SSP-RK4 has 10 stages, at a CFL number of $3.0$ it is still more efficient than RK44 (4 stages at a CFL number of $0.8$). A similar conclusion for the linear advection equation can be found in \cite{article:GoKeSh09}.
Another important feature of SSP-RK4 is its low-storage property, which is  greatly advantageous for 3D simulations and GPU implementations.

One drawback of using a 10-stage RK method versus a lower-stage RK method is that it may
become inefficient if it is used in an adaptive mesh refinement framework such as the one proposed
in \cite{article:ShQiCh}. The main difficulty is that as the number of stages increases, the number of ghost cells on each grid patch is also increased, which results in more communication across grid
patches. In the current work, we are only considering a single grid (i.e., no adaptive mesh refinement),
and thus the 10-stage SSP-RK works well. We will consider issues related to AMR implementations of the proposed scheme in future work.

\section{FD-WENO spatial discretization of the ideal MHD equations}
\label{sdmhd}
In this section we describe the semi-discrete finite difference weighted essentially non-oscillatory (FD-WENO) scheme
that comprises the {\it base scheme} in the constrained transport framework described in \S \ref{ctmtd}. 
Our method of choice is the finite difference WENO method developed by Jiang and Wu \cite{article:JiWu99}.
We will refer to this method as the WENO-HCL\footnote{WENO-HCL := Weighted Essentially Non-Oscillatory for Hyperbolic Conservation Laws.} scheme.
In what follows, we describe the basic WENO-HCL scheme in one space dimension for the ideal MHD equations.
We then briefly discuss the straightforward extension to higher dimensions.

We write the MHD system $\eqref{MHD}$ in 1D as follows:
\begin{align}
\label{eqn:conslaw}
q_{,t} + f(q)_{,x} = 0,
\end{align}
where
\begin{gather}
q =  \left(\rho, \rho u^1, \rho u^2, \rho u^3, \mathcal{E}, B^1, B^2,B^3\right), \\
\begin{split}
f(q)  &=  \biggl( \rho u^1, \rho u^1 u^1 + p + \frac{1}{2} \|\B\|^2 - B^1 B^1, \rho u^1 u^2 -  B^1 B^2,  \rho u^1 u^3 - B^1 B^3, \\ 
  & \qquad u^1 \left(\mathcal{E}+p+\frac{1}{2}\|\B\|^2 \right)- B^1(\u \cdot \B), 0, u^1 B^2 - u^2 B^1, u^1 B^3 - u^3 B^1 \biggr).
\end{split}
\end{gather}
For convenience, we also introduce
\begin{align}
u =(\rho,u^1,u^2,u^3,p,B^1,B^2,B^3)
\end{align}
to denote the vector of primitive variables. 

Due to the hyperbolicity of the MHD systems, the flux Jacobian matrix $f_{,q}$ has a
spectral decomposition of the form
\begin{align}
f_{,q} = R \Lambda L,
\end{align}
where $\Lambda$ is the diagonal matrix of real eigenvalues, $R$ is the matrix of right eigenvectors and $L = R^{-1}$ is the matrix of left eigenvectors.

We consider the problem on a uniform grid with $N+1$ grid points as follows:
\begin{align}
a = x_{0} < x_{1} < \dots < x_{N} = b,
\end{align}
and let $q_i(t)$ denote the  approximate solution of the MHD system at the point
 $x = x_i$. The WENO-HCL scheme for system \eqref{eqn:conslaw} can be written in the following flux-difference form:
\begin{align}
\label{fdcf}
	q'_i(t) = \frac{1}{\Delta x} \left( \hat{f}_{i+\frac{1}{2}} - \hat{f}_{i-\frac{1}{2}} \right).
\end{align}
To obtain the numerical flux, $\hat{f}_{i+\frac{1}{2}}$, in the above semi-discrete form, the following WENO 
procedure is used:
\begin{enumerate}
\item Compute the physical flux at each grid point:
\begin{align}
f_i = f(q_i).
\end{align}
\item At each $x_{i+\frac{1}{2}}$: 
\begin{enumerate}
\item Compute the average state $u_{i+\half}$ in the primitive variables:
\begin{align}
\label{eqn:arthm_ave}
	u_{i+\frac{1}{2}} = \frac{1}{2}\left( u_i + u_{i+1} \right).
\end{align}
\item Compute the right and left eigenvectors of the flux Jacobian matrix, $f_{,q}$, at $x = x_{i+\frac{1}{2}}$:
\begin{align}
	R_{i+\frac{1}{2}} = R\left( u_{i+\frac{1}{2}} \right) \quad \text{and} \quad
	L_{i+\frac{1}{2}} = L\left( u_{i+\frac{1}{2}} \right),
\end{align}
where $L_{i+\frac{1}{2}} = R^{-1}_{i+\frac{1}{2}}$.
\item Project the solution and physical flux into the right eigenvector space:
\begin{align}
\begin{split}
	V_{ j } = L_{i+\half} \,  q_{j} \quad \text{and} \quad
	G_{ j } = L_{i+\half} \,  f_{j},
\end{split}
\end{align}
for all $j$ in the numerical stencil associated with $x = x_{i+\half}$.
In the case of the $5^{\text{th}}$-order FD-WENO scheme: 
$j = i-2, i-1, i, i+1, i+2, i+3$.
\item Perform a Lax-Friedrichs flux vector splitting for each component of the characteristic variables. Specifically, assume that the $m^\text{th}$ components of $V_j$ and $G_j$ are $v_j$ and $g_j$, respectively, then compute
\begin{align}
	g^{\pm}_{j} = \frac{1}{2}\left( g_{j} \pm \alpha^{(m)} v_{j} \right),
\end{align}
where
\begin{align} 
\alpha^{(m)} = \max_{k} \left| \lambda^{(m)}( q_{k} ) \right|
\end{align}
is the maximal wave speed of the $m^\text{th}$ component of characteristic variables over all grid points. Note that the eight
eigenvalues for ideal MHD are given in \S\ref{hypermhd}.
\item Perform a WENO reconstruction on each of the computed flux components $g^{\pm}_{j}$ to obtain the corresponding component of the numerical flux.  If we let $\Phi_{\text{WENO}5}$ denote the $5^{\text{th}}$-order WENO reconstruction
operator (see \ref{sec:appendix} for a detailed description), then the flux is computed as
follows:
\begin{align}
	\hat{g}^+_{i+1/2} & = \Phi_{\text{WENO}5}\left( g^+_{i-2}, g^+_{i-1}, g^+_{i}, g^+_{i+1}, g^+_{i+2} \right), \\
	\hat{g}^-_{i+1/2} & = \Phi_{\text{WENO}5}\left( g^-_{i+3}, g^-_{i+2}, g^-_{i+1}, g^-_{i}, g^-_{i-1} \right).
\end{align}
Then set
\begin{align}
\hat{g}_{i+\frac{1}{2}} = \hat{g}^+_{i+\frac{1}{2}}  + \hat{g}^-_{i+\frac{1}{2}},
\end{align}
where $\hat{g}_{i+\half}$ is the $m^\text{th}$ component of $\hat{G}_{i+\frac{1}{2}}$.
\item Project the numerical flux back to the conserved variables
\begin{align}
\hat{f}_{ i+\frac{1}{2} } = R_{i+\frac{1}{2}} \, \hat{G}_{i+\frac{1}{2}}.
\end{align}
\end{enumerate}
\end{enumerate}

\begin{remark}
In Step (a), although one could define the average state at $x_{i+\frac{1}{2}}$ using the Roe averages developed by
 Cargo and Callice \cite{article:CaCa97}, we instead define the state at $x_{i+\frac{1}{2}}$ via simple arithmetic averages of the primitive variables (equation \eqref{eqn:arthm_ave}). The arithmetic averages are computationally less expensive to evaluate than the Roe averages and produce good numerical results in practice. 
It was  pointed out in \cite{article:JiWu99} that there is little difference in the numerical results when different approaches for defining the half-grid state are used in the base WENO scheme.
\end{remark}

\begin{remark}
In Step (b) there are several different versions of right eigenvectors scalings \cite{barth1999,article:BrWu88,article:Po94,article:Po99}. In this work we make use of
 the eigenvector scaling based on entropy variables proposed by Barth \cite{barth1999}. This approach
is advantageous in that it is relatively simple to implement and gives the optimal direction-independent matrix norm of the eigenvector matrix.
\end{remark}

\begin{remark}
In Step (d) we use a {\it global} Lax-Friedrichs flux splitting, meaning that the $\alpha^{(m)}$'s
are computed as the maximum of the $m^{\text{th}}$ eigenvalue over the entire mesh.
One could just as easily use a {\it local} Lax-Friedrichs flux splitting and only maximize the
eigenvalue over the stencil on which the flux is defined. In all of the
numerical test problems attempted in this work, we found no significant differences
between the global and local approaches. In some applications where the eigenvalue
changes dramatically in different regions of the computational domain it may be 
advantageous to switch to the local Lax-Friedrichs approach. 
\end{remark}


In the 1D case we find that the above WENO-HCL scheme applied to MHD
can produce high-order accurate solutions for smooth problems
and can accurately capture shocks without producing unphysical oscillations around discontinuities.
The scheme as described so far can be easily extended to higher dimensions, simply by applying the
FD-WENO definition of the numerical fluxes dimension-by-dimension.

The multi-dimensional version of the method described in this section serves as the {\it base scheme} for our proposed 
constrained transport method for ideal MHD. However, as has been well-documented in the literature,
direct application of only the base scheme will lead to divergence errors in the magnetic field, which in turn will lead
to numerical instabilities (e.g., see  Example~$\ref{section:2drs}$ in Section~\ref{numer}). 
In order to overcome this problem, we also need to directly evolve the magnetic potential as outlined in \S \ref{ctmtd}.
In the next section we show how to modify the WENO-HCL scheme to create a high-order accurate
numerical update for the magnetic vector potential equation, which will then  be used
to correct the magnetic field that is predicted by the WENO-HCL base scheme (see \S \ref{ctmtd}
for the full outline of a single stage of the proposed constrained transport scheme).

\section{FD-WENO spatial discretization of the magnetic potential equation}
\label{sdmpe}
In this section we discuss a novel approach for discretizing the magnetic potential equations in 2D and 3D. 
There are two main challenges in obtaining such discretizations:
(1) we must design a high-order finite difference method capable of solving the non-conservative and weakly hyperbolic
system that the magnetic potential satisfies; and (2) we must design appropriate limiting strategies that
act on the update for the magnetic potential, but which control unphysical oscillations in the magnetic field. 
The approach we develop is a modification of the WENO method of Jiang and Peng \cite{article:jiang2000},
which was designed for Hamilton--Jacobi equations. 
We begin by describing the 1D version of WENO scheme of Jiang and Peng \cite{article:jiang2000},
then show how to modify this approach to solve the scalar 2D magnetic potential equation \eqref{2dpte}, and finally
describe how to generalize this to the more complicated systems 3D magnetic potential equation \eqref{3dpte}.

\subsection{WENO for 1D Hamilton--Jacobi}
Consider a 1D Hamilton--Jacobi equation of the form
\begin{align}
\label{2dhj}
	q_{,t} + H\left(t,x,q,q_{,x} \right) = 0,
\end{align}
where $H$ is the Hamiltonian.
Jiang and Peng \cite{article:jiang2000} developed a semi-discrete approximation to
\eqref{2dhj} of the following form:
\begin{align}
	\frac{d q_{i}(t)}{dt} = -\hat{H}\left( t,x_i,q_{i},q^-_{,x i},q^+_{,x i} \right),
\end{align}
where $\hat{H}$ is the numerical Hamiltonian and is consistent with $H$ in the following sense:
\begin{align}
\hat{H}\left(t,x,q,u,u\right) = H\left(t,x,q,u\right),
\end{align}
and $q^-_{,x i}$ and $q^+_{,x i}$ are left and right-sided approximations of $q_{,x}$ at $x=x_i$.
The values of  $q^-_{,x i}$ and $q^+_{,x i}$ are obtained by performing WENO reconstruction as follows:
\begin{align}
\label{dqweno1}
q^-_{,x i} & = {\Phi}_{\text{WENO}5}\left( \frac{\Dphi_{i-3}}{\Dx}, \frac{\Dphi_{i-2}}{\Dx}, \frac{\Dphi_{i-1}}{\Dx}, \frac{\Dphi_{i}}{\Dx}, \frac{\Dphi_{i+1}}{\Dx} \right), \\
\label{dqweno2}
q^+_{,x i} & = {\Phi}_{\text{WENO}5}\left( \frac{\Dphi_{i+2}}{\Dx}, \frac{\Dphi_{i+1}}{\Dx}, \frac{\Dphi_{i}}{\Dx}, \frac{\Dphi_{i-1}}{\Dx}, \frac{\Dphi_{i-2}}{\Dx} \right),
\end{align}
where
\begin{align}
\Dphi_{ij} := q_{i+1 j} - q_{ij},
\end{align}
and ${\Phi}_{\text{WENO}5}$ uses the same formula as the one for WENO-HCL (see \ref{sec:appendix}). The difference here is that the reconstructions $\eqref{dqweno1}$ and $\eqref{dqweno2}$ are applied to the central derivative of the solution $q$, while the reconstruction in WENO-HCL is applied to the fluxes on grid points. The new reconstruction helps us control unphysical oscillations in $q_{,x}$ not $q$.
This is in an important distinction since
with Hamilton--Jacobi we are solving for a potential, the derivatives of which produce a physical
variable.

As described in detail in \ref{sec:appendix}, the WENO reconstruction formulas, $\Phi_{\text{WENO}5}$, depend
on smoothness indicators, $\beta$, that control how much weight to assign the different finite difference
stencils. In the standard WENO-HCL framework, the weights are chosen so as to control unphysical
oscillations in the conserved variables. In WENO-HCL the smoothness indicator is computed as follows \cite{article:JiShu96}:
\begin{align}
\label{eqn:sm-indicator}
\beta_j = \sum_{\ell=1}^{k} \Delta x^{2\ell-1} \int_{x_{i-\frac{1}{2}}}^{x_{i+\frac{1}{2}}} \left(\frac{d^{\ell}}{dx^{\ell}} p_j(x)\right)^2 \, dx,
\end{align}
where $p_j$ is an interpolating polynomial of the values $q_i$  in some stencil and $k$ is the degree of $p_j$. Note that
 the smoothness indicator is computed by including the normalized total variation of the first derivatives of $p_j$ (i.e., the $\ell=1$ term), leading to an essentially non-oscillatory solution $q_i$. However, $\beta_j$ could be dominated by the total variation of $\frac{d}{dx} p_j(x)$, which is not important in controlling the oscillation of $q_{,x}$.
 
Based on the above observation about WENO reconstruction, we realize if the reconstruction is applied to $\frac{\Dphi_{i}}{\Dx}$ as in $\eqref{dqweno1}$ and $\eqref{dqweno2}$, the new $p_j$ becomes an interpolating polynomial of $q_{,x}$. The same smoothness indicator formula $\eqref{eqn:sm-indicator}$ evaluates the smoothness of the interpolating polynomial of $q_{,x}$ in this case. So if $q_{,x}$ is not smooth, this procedure will approximate the derivative $q_{,xi}$ by essentially using the stencil that has the smoothest derivative. In other words, the oscillations in $q_{,x}$ can be controlled. 
A similar idea is used in the method of Rossmanith \cite{article:Ro04b}, where a TVD limiter is applied to wave differences 
instead of waves so as to control the oscillation in the computed solution derivatives.



Finally, in order to evaluate the numerical Hamiltonian, $\hat{H}$, Jiang and Peng \cite{article:jiang2000}
introduced a Lax-Friedrichs-type definition:
\begin{align}
\label{wenohj}
\begin{split}
\hat{H}\left(t,x,u^-,u^+ \right) = & H\left(t,x,\frac{u^-+u^+}{2} \right) 
 - \alpha \left( \frac{u^+-u^-}{2} \right),
\end{split}
\end{align}
where
\begin{align}
\alpha = \max_{\substack{u \, \in \, I(u^-, u^+)}} \left|H_{,u}(t,x,q,u)\right|,
\end{align}
where $I(u^-, u^+)$ is the interval between $u^-$ and $u^+$.

We refer to the scheme discussed in this section as the WENO-HJ scheme.
To compare the WENO-HJ scheme with the WENO-HCL scheme, we consider a 
simple test problem on which both WENO-HJ and WENO-HCL can be applied.
We consider the 1D linear constant coefficient advection equation,
\begin{equation}
q_{,t} + q_{,x} = 0,
\end{equation}
 on $x \in [0,1]$ with the periodic boundary condition and the following piecewise linear initial condition:
\begin{align}
q(0,x) = \begin{cases}
0 & \mbox{if \quad} 0.00 \le x \le 0.25, \\ 
(x-0.25)/0.075 & \mbox{if \quad} 0.25 \le x \le 0.40, \\
2 & \mbox{if \quad} 0.40 \le x \le 0.60, \\
(0.75-x)/0.075 & \mbox{if \quad} 0.60 \le x \le 0.75, \\ 
0 & \mbox{if \quad} 0.75 \le x \le 1.00.
\end{cases}
\end{align} 
This problem was considered by Rossmanith \cite{article:Ro04b}, where it was also used
to test a limiter especially designed to control oscillations in the derivative of $q$.

The solutions and their numerical derivatives computed by WENO-HCL and WENO-HJ schemes are presented in 
Figure~\ref{1dadcp}. Both approaches use $5^{\text{th}}$-order WENO reconstruction in the spatial discretization and SSP-RK4 for the time integrator. We use a CFL number of 1.0 and compute the solution to $t = 4$. 
Shown in this figure  are \ref{1dadcp}(a)  the solution obtained by the WENO-HCL scheme on a mesh with $N=300$,  \ref{1dadcp}(b) the derivative of this solution as computed with a
      $4^{\text{th}}$-order central difference approximation,  
      \ref{1dadcp}(c)  the solution obtained by the WENO-HJ scheme on a mesh with $N=300$,  and
      \ref{1dadcp}(d) the derivative of this solution as computed with a
      $4^{\text{th}}$-order central difference approximation.
Although both solutions agree with the exact solution very well, the computed derivative $q_{,x}(4,x)$ of WENO-HCL is much more oscillatory than that of WENO-HJ. 
The proposed WENO-HJ scheme is able to control unphysical oscillations both in the solution and its derivative.  
The result of our new approach is comparable to that of existing finite volume approaches 
\cite{article:HeRoTa10,article:helzel2013,article:Ro04b}.

\subsection{2D magnetic potential equation}
\label{2dmp}
In the CT framework described in \S \ref{ctmtd}, during Step 1 we must update the magnetic
potential by solving a discrete version of
\begin{equation}
\label{eqn:a3-ham1}
	A^3_{,t} + u^{1}(x,y) \, A^3_{,x} + u^{2}(x,y) \, A^3_{,y} = 0,
\end{equation}
where, as described in Section~\ref{ctmtd}, the velocity components are given functions
from the previous time  step (or time stage in the case of higher-order time-stepping).
Because  $u^1$ and $u^2$ are given, we can view \eqref{eqn:a3-ham1} as a Hamilton--Jacobi
equation:
\begin{equation}
   A^3_{,t} + H\left(x,y,A^3_{,x}, A^3_{,y} \right) = 0, 
\end{equation}
with Hamiltonian:
\begin{equation}
   H\left(x,y,A^3_{,x}, A^3_{,y} \right) = u^{1}(x,y) \, A^3_{,x} + u^{2}(x,y) \, A^3_{,y}.
\end{equation}

To this equation we can directly apply a two-dimensional version of the WENO-HJ scheme
described above (just as with WENO-HCL, the 2D version of WENO-HJ is simply a direction-by-direction
version of the 1D scheme). The 2D semi-discrete WENO-HJ can be written as
\begin{align}
\begin{split}
\label{eqn:semi-a3}
 \frac{d A^3_{ij}(t)}{dt} & = - \hat{H}\left(A^{3-}_{,xij},A^{3+}_{,xij},A^{3-}_{,yij},A^{3+}_{,yij}\right) \\
 = & -u^1_{ij} \left( \frac{A^{3-}_{,xij}+A^{3+}_{,xij}}{2} \right) - u^2_{ij} \left( \frac{A^{3-}_{,yij}+A^{3+}_{,yij}}{2} \right)  \\
& + \alpha^1 \left( \frac{A^{3+}_{,xij}-A^{3-}_{,xij}}{2} \right) + \alpha^2 \left( \frac{A^{3+}_{,yij}-A^{3-}_{,yij}}{2} \right),
\end{split}
\end{align}
where 
\begin{align*}
\alpha^1 = \max_{ij} \left|u^1_{ij}\right| \quad \text{and} \quad
\alpha^2 = \max_{ij} \left|u^2_{ij}\right|.
\end{align*}
The approximations $A^{3\pm}_{,xij}$ and $A^{3\pm}_{,yij}$ are calculated
with formulas analogous to \eqref{dqweno1} and \eqref{dqweno2}. 
$\alpha^1$ and $\alpha^2$ are chosen as the maximal value over all grid points based on a similar idea of the Lax-Friedrichs flux splitting in \S \ref{sdmhd}. 
We remark here the scheme $\eqref{eqn:semi-a3}$ with this global $\alpha^m$ can be too dissipative for certain pure linear Hamilton--Jacobi equations. 
Another obvious choice is to evaluate $\alpha^m$ by taking the maximal value on the local stencil. 
Although this local version of $\alpha^m$ can be much less dissipative for a certain pure Hamilton--Jacobi equation, 
in numerical experiments for MHD we find that the differences between the local and global approaches are negligible. 
Therefore, we will only present the numerical results by the global version of $\alpha^m$ in the numerical examples section (\S \ref{numer}).

Except for the last remaining detail about how the discrete curl of $A$ is computed (see \S \ref{sdcurl}), 
this version of the WENO-HJ scheme coupled with the WENO-HCL scheme as the base scheme completes our 2D finite difference WENO constrained transport method. In the numerical examples section, \S \ref{numer},
we will refer to this full scheme as WENO-CT2D.

\subsection{3D magnetic potential equation}
\label{3dmp}
The evolution equation of the magnetic potential \eqref{3dpte} in 3D is significantly 
different from the evolution equation of the scalar potential \eqref{2dpte} in 2D, 
and hence the scheme discussed in Section \ref{2dmp} cannot be used directly. 
As pointed out in \S \ref{mhd-eq-mq3d}, a key difficulty is the weak hyperbolicity of system \eqref{3dpte}.
 Helzel et al. \cite{article:HeRoTa10} found that the weak hyperbolicity of \eqref{3dpte} is
 only an artifact of freezing the velocity in time, i.e., the full MHD system is still strongly hyperbolic. 
 Furthermore, they found that the magnetic vector potential system can be solved by an operator split
 finite volume scheme with an additional limiting strategy added in certain directions. 
 Helzel et al. \cite{article:helzel2013} handled the weakly hyperblic system \eqref{3dpte}
 through a path-conserving finite volume WENO scheme without appeal to operator splitting. 
 
 In order to explain the scheme advocated in this work, we 
 take inspiration from the operator split method of Helzel et al. \cite{article:HeRoTa10}
and separate system \eqref{3dpte} into two sub-problems:
\begin{description}
\item [{\bf Sub-problem 1:}]
\begin{align}
\label{sub1}
\begin{split}
A^1_{,t} + u^2 A^1_{,y} + u^3 A^1_{,z} = 0,\\
A^2_{,t} + u^1 A^2_{,x} + u^3 A^2_{,z} = 0,\\
A^3_{,t} + u^1 A^3_{,x} + u^2 A^3_{,y} = 0.
\end{split}
\end{align}
\item [{\bf Sub-problem 2:}]
\begin{align}
\label{sub2}
\begin{split}
A^1_{,t} - u^2 A^2_{,x} - u^3 A^3_{,x} = 0,\\
A^2_{,t} - u^1 A^1_{,y} - u^3 A^3_{,y} = 0,\\
A^3_{,t} - u^1 A^1_{,z} - u^2 A^2_{,z} = 0.
\end{split}
\end{align}
\end{description}
We emphasis here that our final scheme will not contain any operator splitting, and
that the division of the magnetic potential evolution equation into the above two
sub-problems is only for the purpose of exposition.

The first sub-problem is a combination of three independent evolution equations, each of which has the same mathematical
form as the 2D scalar evolution equation \eqref{eqn:a3-ham1}. Furthermore, this sub-problem is strongly hyperbolic.
Thus, at least for this sub-problem, we can simply use the WENO-HJ scheme described in Section~\ref{2dmp} to solve 
these three equations independently.

The second sub-problem, \eqref{sub2}, is only weakly hyperbolic.
For this problem we apply a WENO finite difference discretization
using arithmetic averages to define solution derivatives at grid points.
For instance, for the first component $A^1$ in \eqref{sub2}, the semi-discrete form becomes
\begin{align}
\label{smsb}
 \frac{d}{dt}
	A^1_{ijk}(t)
 = u^2_{ijk} \left( \frac{A^{2-}_{,xijk} + A^{2+}_{,xijk}}{2} \right) + u^3_{ijk} \left( \frac{A^{3-}_{,xijk} + A^{3+}_{,xijk}}{2} \right),
\end{align}
where $A^{m\pm}_{,xijk}$ again uses the WENO reconstruction given by \eqref{dqweno1}--\eqref{dqweno2}.
The semi-discrete forms for the other components in \eqref{sub2} are similar.

Note that semi-discrete formula \eqref{smsb} lacks the numerical dissipation
terms found in \eqref{eqn:semi-a3}. 
This is due to the fact that system \eqref{sub2} (by itself)
does not represent a transport equation.
Therefore, the above described discretizations for \eqref{sub1} and \eqref{sub2} generally do not 
introduce sufficient numerical resistivity in order to control unphysical oscillations in the 
magnetic field for a 3D problem.
To be more precise, when solving system \eqref{sub1}, artificial resistivity is introduced
from the WENO upwinding procedure (see formula \eqref{eqn:semi-a3}), but only in 2 of the 3
coordinate directions (e.g., for $A^3$ there is artificial resistivity in the $x$ and $y$-directions). 
When solving system \eqref{sub2}, no additional artificial resistivity is introduced (e.g., see 
\eqref{smsb}).  This lack of numerical dissipation
was also pointed out by Helzel et al. \cite{article:HeRoTa10,article:helzel2013}; they introduced 
explicit artificial resistivity terms into the magnetic vector potential equation. We follow a similar
approach by modifying sub-problem  \eqref{sub2} as follows: 
\begin{description}
\item[{\bf Sub-problem 2 with artificial resistivity:}]
\begin{align}
\label{sbart}
\begin{split}
A^1_{,t} - u^2 A^2_{,x} - u^3 A^3_{,x} = \varepsilon^1 A^1_{,x,x}, \\
A^2_{,t} - u^1 A^1_{,y} - u^3 A^3_{,y} = \varepsilon^2 A^2_{,y,y}, \\
A^3_{,t} - u^1 A^1_{,z} - u^2 A^2_{,z} = \varepsilon^3 A^3_{,z,z}.
\end{split}
\end{align}
\end{description}
These additional terms give us artificial resistivity in the missing directions
(e.g., the equation for $A^3$ now has an artificial resistivity term in the $z$-direction).
In the above expression, the artificial resistivity is take to be of the following form:
\begin{align}
\label{eqn:art_res_1}
\varepsilon^1 = 2 \nu \gamma^1 \frac{\Delta x^2}{\delta + \Delta t},
\end{align}
where $0 \le \delta \ll 1$ is small parameter that can be set to ensure that $\varepsilon^1$ remains
bounded as $\Delta t \rightarrow 0^+$, $\gamma^1$ is the smoothness indicator of $A^1$, and $\nu$ is a constant used to control the magnitude of the 
artificial resistivity.

 In all the simulations presented in this work, we take $\delta = 0$ ($\Delta t$ and $\Delta x$
 have the same order of magnitude for all the problems considered in this work).
 The smoothness indicator $\gamma^1$ is computed as follows:
\begin{align}
\label{eqn:art_res_2}
\gamma^1_{ijk} = \left| \frac{a^-}{a^-+a^+} - \frac{1}{2} \right|,
\end{align}
where
\begin{align}
\label{eqn:art_res_3}
a^- =\left\{ \epsilon + \left( \Dx \, A^{1-}_{,xijk} \right)^2 \right\}^{-2} \quad \text{and} \quad
a^+ =\left\{ \epsilon + \left( \Dx \, A^{1+}_{,xijk} \right)^2 \right\}^{-2},
\end{align}
and $\epsilon$ is taken to be $10^{-8}$ in all of our numerical computations. Here $a^-$ and $a^+$ are used to indicate the smoothness of $A^1$ in each of the $-$ and $+$ WENO stencils, respectively. The artificial resistivity parameters $\epsilon^2$ and $\epsilon^3$ in the other directions can be computed in analogous ways.

The smoothness indicator $\gamma^i$ is designed such that sufficient artificial resistivity is introduced to avoid
spurious oscillations in the derivatives of $A^i$ when $A^i$ is non-smooth, and  high-order accuracy of the scheme is
maintained when $A^i$ is smooth. For the case when $A^1_{,x}$ is smooth:
\begin{align*}
 A^{1-}_{,xijk} -  A^{1+}_{,xijk} = O(\Dx^5) \quad \text{and} \quad \gamma^1_{ijk} = O(\Dx^5).
\end{align*}
In this case the artificial resistivity term in \eqref{sbart} will be of $O(\Dx^6)$, which will not destroy
 the $5^{\text{th}}$-order spatial accuracy of the scheme. For the case when $A^1_{,x}$ is non-smooth:
\begin{align*}
 A^{1-}_{,xijk} -  A^{1+}_{,xijk} = O(1) \quad \text{and} \quad \gamma^1_{ijk} \approx \frac{1}{2},
\end{align*}
which indicates that numerical resistivity should be added. For both the smooth and non-smooth cases, we note that 
$\gamma^1 < \frac{1}{2}$, which means that for forward Euler time stepping, $\nu$ in the range of $[0,0.5]$ will
guarantee that the numerical scheme will be stable up to $\text{CFL}=1$. For the fourth-order 10-stage SSP-RK4 time-stepping
scheme, we found that $\nu$ in the range of $[0.02,0.2]$ seems to satisfactorily control the unphysical
oscillations in 3D problems.

The constrained transport method that we advocate in this work is a method of lines approach, and thus is not consistent with the operator splitting approach. However, through numerical experimentation, we discovered that operator splitting is not
necessary to obtain accurate and stable solutions, as long as the above artificial resistivity limiting strategy is included
in the time evolution. In order to write out the final scheme as advocated, consider for brevity only the first equation
in the magnetic vector potential system with artificial resistivity:
\begin{align} 
 A^1_{,t} - u^2 A^2_{,x} - u^3A^3_{,x} + u^2 A^1_{,y} + u^3 A^1_{,z} = \epsilon^1 A^1_{,x,x}.
\end{align}
Using the above discussion about sub-problems 1 and 2 as a guide, we arrive at the following unsplit
semi-discrete form for the full $A^1$ evolution equation:
\begin{align}
\label{eqn:dA1dt}
\begin{split}
 \frac{d A^1_{ijk}(t)}{dt} = & u^2_{ijk} \left( \frac{A^{2-}_{,xijk} + A^{2+}_{,xijk}}{2} \right) + u^3_{ijk} \left( \frac{A^{3-}_{,xijk} + A^{3+}_{,xijk}}{2} \right) \\
 & + 2 \nu \gamma^1 \left ( \frac{A^{1}_{i-1jk} - 2A^{1}_{ijk} + A^{1}_{i+1jk}}{\delta + \Delta t} \right)\\
 & -u^2_{ijk} \left(\frac{A^{1-}_{,yijk}+A^{1+}_{,yijk}}{2} \right) - u^3_{ijk} \left( \frac{A^{1-}_{,zijk}+A^{1+}_{,zijk}}{2}\right) \\
& + \alpha^2 \, \left( \frac{A^{1+}_{,yijk}-A^{1-}_{,yijk}}{2} \right) + \alpha^3 \, \left (\frac{A^{1+}_{,zijk}-A^{1-}_{,zijk}}{2} \right), \end{split}
\end{align}
where 
\begin{align*}
\alpha^2 = \max_{i,j,k} \left|u^2_{ijk}\right|, \quad \text{and} \quad
\alpha^3 = \max_{i,j,k} \left|u^3_{ijk}\right|.
\end{align*}
The semi-discrete forms for $A^2$ and $A^3$ of the system have analogous forms. For brevity we omit these
formulas.

Except for the last remaining detail about how the discrete curl of $A$ is computed (see \S \ref{sdcurl}), 
this version of the WENO-HJ scheme coupled with the WENO-HCL scheme as the base scheme completes our 3D finite difference WENO constrained transport method. In the numerical examples section, \S \ref{numer},
we will refer to this full scheme as WENO-CT3D.

Finally, we note that the artificial resistivity terms included in the semi-discrete equation \eqref{eqn:dA1dt} with \eqref{eqn:art_res_1},  \eqref{eqn:art_res_2},  and 
\eqref{eqn:art_res_3} are specifically designed for solving the ideal MHD equations \eqref{MHD}.
In future work we will consider non-ideal corrections, including physical resistivity and the Hall
term. In these non-ideal cases, modifications will have to be made to  the artificial resistivity
terms advocated in this work.

\section{Central finite difference discretization of $\nabla \times \Av$}
\label{sdcurl}
During each stage of our CT algorithm, a discrete curl operator is applied to the magnetic potential to give a divergence-free magnetic field. In this section we describe the approach to approximate the curl operator and discuss its important properties.

\subsection{Curl in 2D}
We look for a discrete version of the 2D curl given by \eqref{2dcurl} of the following form:
\begin{align}
  B^1_{ij} := D^{\, y}_{ij} \, A^3 \quad \text{and} \quad B^2_{ij} := -D^{\, x}_{ij} \, A^3,
\end{align}
where $D^{\, x}$ and $D^{\, y}$ are discrete versions of the operators $\partial_x$ and $\partial_y$.
In particular, we look for discrete operators $D^{x}$ and $D^{y}$ with the property that
\begin{align}
   \nabla \cdot \B_{ij} := D^{\, x}_{ij} \, B^1 + D^{\, y}_{ij} \, B^2 = D^{\, x}_{ij} \, D^{\, y}_{ij} \, A^3 - D^{\, y}_{ij} \, D^{\, x}_{ij} \, A^3 = 0,
\end{align}
which means that we satisfy a discrete divergence-free condition.
In the second order accurate, unstaggered, CT methods developed by  Helzel et al. \cite{article:HeRoTa10}, Rossmanith \cite{article:Ro04b}, and T\'oth \cite{article:To00}, the obvious choice for $D^{x}_{ij}$ and $D^{y}_{ij}$
are $2^{\text{nd}}$-order central finite differences.
In this work, in order to maintain high-order accuracy, we replace
$2^{\text{nd}}$-order central differences with $4^{\text{th}}$-order central finite differences:
\begin{align}
 D^{\, x}_{ij} \, A &:= \frac{1}{12\Delta x} \left( A_{i-2 \, j} - 8 A_{i-1 \, j} + 8 A_{i+1 \, j} - A_{i+2 \, j} \right), \\
 D^{\, y}_{ij} \, A &:= \frac{1}{12\Delta y} \left( A_{i \, j-2} - 8 A_{i \, j-1} + 8 A_{i \, j+1} - A_{i \, j+2} \right).
\end{align}

\subsection{Curl in 3D}
We look for a discrete version of the 3D curl of the following form:
\begin{align}
  B^1_{ijk} &:= D^{\, y}_{ijk} \, A^3 - D^{\, z}_{ijk} \, A^2, \\
  B^2_{ijk} &:= D^{\, z}_{ijk} \, A^1 - D^{\, x}_{ijk} \, A^3, \\
  B^3_{ijk} &:= D^{\, x}_{ijk} \, A^2 - D^{\, y}_{ijk} \, A^1.
\end{align}
where $D^{\, x}$, $D^{\, y}$, and $D^{\, z}$ are discrete versions of the operators $\partial_x$, $\partial_y$, and $\partial_z$.
In particular, we look for discrete operators $D^{x}$, $D^{y}$, and $D^{z}$ with the property that
\begin{equation}
\begin{split}
   \nabla \cdot \B_{ijk} := & \, \, D^{\, x}_{ijk} \, B^1 + D^{\, y}_{ijk} \, B^2 + D^{\, z}_{ijk} \, B^3 \\
   = & \, \, D^{\, x}_{ijk} \, D^{\, y}_{ijk} \, A^3 - D^{\, x}_{ijk} \, D^{\, z}_{ijk} \, A^2 
   + D^{\, y}_{ijk} \, D^{\, z}_{ijk} \, A^1 \\ - & \, \, D^{\, y}_{ijk} \, D^{\, x}_{ijk} \, A^3 + D^{\, z}_{ijk} \, D^{\, x}_{ijk} \, A^2
   - D^{\, z}_{ijk} \, D^{\, y}_{ijk} \, A^1 = 0,
 \end{split}
\end{equation}
which means that we satisfy a discrete divergence-free condition.
To achieve higher-order accuracy we again use $4^{\text{th}}$-order central finite differences:
\begin{align}
 D^{\, x}_{ijk} \, A &:= \frac{1}{12\Delta x} \left( A_{i-2 \, j k} - 8 A_{i-1 \, j k} + 8 A_{i+1 \, j k} - A_{i+2 \, j k} \right), \\
 D^{\, y}_{ijk} \, A &:= \frac{1}{12\Delta y} \left( A_{i \, j-2 \, k} - 8 A_{i \, j-1 \, k} + 8 A_{i \, j+1 \, k} - A_{i \, j+2 \, k} \right), \\
 D^{\, z}_{ijk} \, A &:= \frac{1}{12\Delta z} \left( A_{i \, j \, k-2} - 8 A_{i \, j \, k-1} + 8 A_{i \, j \, k+1} - A_{i \, j \, k+2} \right).
\end{align}

\subsection{Important properties}
For smooth solutions, the spatial accuracy of our overall scheme will be $4^{\text{th}}$-order accurate.
This fact is confirmed via numerical experiments in Section~\ref{numer}.
Furthermore, for solutions with discontinuities in the magnetic field, the $4^{\text{th}}$-order central 
discretization of the magnetic potential curl will introduce spurious oscillations. However, as we demonstrated
via numerical experiments in Section~\ref{sdmhd}, we are able to control  any unphysical oscillations in the magnetic fields
through  the limiting strategy that was designed in \S \ref{sdmpe} for the WENO-HJ scheme.

Finally, we point out the following property of the proposed scheme:
\begin{clm}
The constrained transport method as described in this work globally conserves the magnetic field from one Runge--Kutta stage to the next.
\end{clm}

\begin{proof}
Using the same idea as the proof of the conservation of $\B$ in Rossmanith \cite{article:Ro04b}, we can show the total amount of each component of $\B$ can be modified only through loss or gain across its physical boundary. Thus the components of the magnetic field are globally conserved. We omit the details of proof here.
\end{proof}


\section{Numerical Results}
\label{numer}

In this section, the 2D and 3D WENO-CT schemes are applied to several MHD problems. First, both the 2D and 3D schemes are tested on the 2D and 3D smooth Alfv\'en wave problems, respectively.
These problems are used to demonstrate that the proposed methods are fourth-order accurate. 
The scheme is also tested on a rotated shock tube problem in order to examine the shock-capturing ability of the method,
as well as to demonstrate the success in controlling divergence errors. Also considered are the
 2D Orszag-Tang vortex problem and  the 2D, 2.5D, and 3D versions of the cloud-shock interaction problem. For all the examples computed below, the gas constant is $\gamma = 5/3$ and the CFL number is $3.0$.

\subsection{Smooth Alfv\'en wave problem}
We first consider 2D and 3D versions of the smooth Alfv\'en wave problem to demonstrate the order of accuracy of the proposed schemes. 

\subsubsection{2D problem}
We perform a convergence study of the 2D scheme for the 2D smooth Alfv\'en wave problem.  The initial conditions and the computational domain for this version are described in detail in several papers
(e.g., \S 6.1.1 on page 3818 of Helzel et al. \cite{article:HeRoTa10}).
The $L_2$-errors and $L_\infty$-errors of the magnetic field and the magnetic scalar potential are shown in Tables~\ref{tab:2d2} and \ref{tab:2dinf}. Fourth-order
convergence rates of all the quantities are observed when $\text{CFL} = 3.0$, which confirms the temporal and spatial order of accuracy. 
\begin{table}
\begin{center}
\begin{Large}
    \begin{tabular}{|c||c||c||c||c|}
    \hline
     {\normalsize {\bf Mesh}} & {\normalsize \bf Error in $B^1$} & {\normalsize \bf Error in $B^2$} & {\normalsize \bf Error in $B^3$} & {\normalsize \bf Error in $A^3$}   \\ \hline\hline
    {\normalsize $16 \times 32$} & {\normalsize $9.727\times 10^{-5}$} & {\normalsize $2.133\times 10^{-4}$} & {\normalsize $2.828\times 10^{-4}$} & {\normalsize $3.028\times 10^{-5}$} \\ \hline
    {\normalsize $32 \times 64$} & {\normalsize $4.072\times 10^{-6}$} & {\normalsize $9.343\times 10^{-6}$} & {\normalsize $9.394\times 10^{-6}$} & {\normalsize $1.365\times 10^{-6}$} \\ \hline
    {\normalsize $64 \times 128$} & {\normalsize $2.020\times 10^{-7}$} & {\normalsize $4.618\times 10^{-7}$}  & {\normalsize $2.881\times 10^{-7}$} & {\normalsize $7.071\times 10^{-8}$} \\ \hline
    {\normalsize $128 \times 256$} & {\normalsize $1.170\times 10^{-8}$} & {\normalsize $2.596\times 10^{-8}$}  & {\normalsize $9.295\times 10^{-9}$} & {\normalsize $4.138\times 10^{-9}$} \\ \hline 
    {\normalsize $256 \times 512$} & {\normalsize $7.150\times 10^{-10}$} & {\normalsize $1.550\times 10^{-9}$}  & {\normalsize $3.395\times 10^{-10}$} & {\normalsize $2.537\times 10^{-10}$} \\ \hline
    {\normalsize $512 \times 1024$} & {\normalsize $4.439\times 10^{-11}$} & {\normalsize $9.496\times 10^{-11}$}  & {\normalsize $1.546\times 10^{-11}$} & {\normalsize $1.577\times 10^{-11}$} \\ \hline    
    \hline
    {\normalsize {\bf EOC}} &  	{\normalsize 4.009}	&    {\normalsize 4.029} & {\normalsize 4.457} &    {\normalsize 4.008} \\
    \hline
    \end{tabular}
    \caption{Convergence study of the 2D Alfv\'en wave for $\text{CFL} = 3.0$. Shown are the $L_{2}$-errors at time $t = 1$ of the magnetic field and magnetic potential as computed by the WENO-CT2D scheme at various grid resolutions. 
    The magnetic field and the magnetic potential values converge at fourth-order accuracy. The experimental order of convergence ({\bf EOC}) is computed by comparing the error for the two finest grids.     \label{tab:2d2}}
    \end{Large}
\end{center}
\end{table}
\begin{table}
\begin{center}
\begin{Large}
    \begin{tabular}{|c||c||c||c||c|}
    \hline
     {\normalsize {\bf Mesh}} & {\normalsize \bf Error in $B^1$} & {\normalsize \bf Error in $B^2$} & {\normalsize \bf Error in $B^3$} & {\normalsize \bf Error in $A^3$}   \\ \hline\hline
    {\normalsize $16 \times 32$} & {\normalsize $2.703\times 10^{-4}$} & {\normalsize $5.793\times 10^{-4}$} & {\normalsize $7.324\times 10^{-4}$} & {\normalsize $6.981\times 10^{-5}$} \\ \hline
    {\normalsize $32 \times 64$} & {\normalsize $1.087\times 10^{-5}$} & {\normalsize $2.467\times 10^{-5}$} & {\normalsize $2.592\times 10^{-5}$} & {\normalsize $3.077\times 10^{-6}$} \\ \hline
    {\normalsize $64 \times 128$} & {\normalsize $4.812\times 10^{-7}$} & {\normalsize $1.091\times 10^{-6}$}  & {\normalsize $7.941\times 10^{-7}$} & {\normalsize $1.564\times 10^{-7}$} \\ \hline
    {\normalsize $128 \times 256$} & {\normalsize $2.729\times 10^{-8}$} & {\normalsize $6.064\times 10^{-8}$}  & {\normalsize $2.501\times 10^{-8}$} & {\normalsize $9.209\times 10^{-9}$} \\ \hline 
    {\normalsize $256 \times 512$} & {\normalsize $1.646\times 10^{-9}$} & {\normalsize $3.615\times 10^{-9}$}  & {\normalsize $8.133\times 10^{-10}$} & {\normalsize $5.670\times 10^{-10}$} \\ \hline
    {\normalsize $512 \times 1024$} & {\normalsize $1.022\times 10^{-10}$} & {\normalsize $2.221\times 10^{-10}$}  & {\normalsize $3.614\times 10^{-11}$} & {\normalsize $3.533\times 10^{-11}$} \\ \hline    
    \hline
    {\normalsize {\bf EOC}} &  	{\normalsize 4.009}	&    {\normalsize 4.025} & {\normalsize 4.492} &    {\normalsize 4.005} \\
    \hline
    \end{tabular}
    \caption{Convergence study of the 2D Alfv\'en wave for $\text{CFL} = 3.0$. Shown are the $L_{\infty}$-errors at time $t = 1$ of the magnetic field and magnetic potential as computed by the WENO-CT2D scheme at various grid resolutions. 
    The magnetic field and the magnetic potential values converge at fourth-order accuracy. The experimental order of convergence ({\bf EOC}) is computed by comparing the error for the two finest grids.     \label{tab:2dinf}}
    \end{Large}
\end{center}
\end{table}

\subsubsection{3D problem}
We also perform a convergence study of the 3D scheme on a 3D variant of the smooth Alfv\'en wave problem. The initial conditions and the computational domain for this version are described in detail in Helzel et al. \cite{article:HeRoTa10}
(page 3819 in \S 6.2.1). The results of the 3D scheme are presented in Tables \ref{tab:3d2} and \ref{tab:3dinf}. Fourth-order accuracy in all components are confirmed by this test problem.

\begin{table}
\begin{center}
\begin{Large}
    \begin{tabular}{|c||c||c||c|}
    \hline
    {\normalsize {\bf Mesh}} & {\normalsize \bf Error in $B^1$} & {\normalsize \bf Error in $B^2$} & {\normalsize \bf Error in $B^3$}    \\ \hline \hline
    {\normalsize $16 \times 32 \times 32$} & {\normalsize $6.918 \times 10^{-5}$}   & {\normalsize $1.156 \times 10^{-4}$} & {\normalsize $1.062 \times 10^{-4}$} \\ \hline
    {\normalsize $32 \times 64 \times 64$} & {\normalsize $2.734 \times 10^{-6}$}   & {\normalsize $4.984 \times 10^{-6}$} & {\normalsize $4.271 \times 10^{-6}$} \\ \hline
    {\normalsize $64 \times 128 \times 128$} & {\normalsize 1.278 $\times 10^{-7}$}  & {\normalsize $2.454 \times 10^{-7}$} & {\normalsize $2.065 \times 10^{-7}$} \\ \hline \hline
    {\normalsize {\bf EOC}} &  {\normalsize 4.419}	&  	{\normalsize 4.344}	&   {\normalsize 4.370}  \\
    \hline
    \end{tabular}

    \bigskip

    \begin{tabular}{|c||c||c||c|}
    \hline
   {\normalsize {\bf Mesh}}  & {\normalsize \bf Error in $A^1$} & {\normalsize \bf Error in $A^2$} & {\normalsize \bf Error in $A^3$}    \\ \hline \hline
    {\normalsize $16 \times 32 \times 32$} & {\normalsize $6.361\times 10^{-6}$}   & {\normalsize $1.422 \times 10^{-5}$} & {\normalsize $1.532 \times 10^{-5}$} \\ \hline
    {\normalsize $32 \times 64 \times 64$} & {\normalsize $3.308\times 10^{-7}$}   & {\normalsize $5.822 \times 10^{-7}$} & {\normalsize $6.314 \times 10^{-7}$} \\ \hline
    {\normalsize $64 \times 128 \times 128$} & {\normalsize $1.841 \times 10^{-8}$} & {\normalsize $2.878 \times 10^{-8}$} & {\normalsize $3.061 \times 10^{-8}$} \\ \hline \hline
    {\normalsize {\bf EOC}} & {\normalsize 4.167} & {\normalsize 4.338} &  {\normalsize 4.367}   \\
    \hline
    \end{tabular}
    \caption{Convergence study of the 3D Alfv\'en wave for $\text{CFL} = 3.0$. Shown are the $L_{2}$-errors at time $t = 1$ of all the magnetic field and magnetic potential components as computed by the WENO-CT3D scheme on various 
    grid resolutions. The magnetic field and magnetic potential values converge at fourth-order accuracy. The experimental order of convergence ({\bf EOC}) is computed by comparing the error for the two finest grids. \label{tab:3d2}}
\end{Large}
\end{center}
\end{table}
\begin{table}
\begin{center}
\begin{Large}
    \begin{tabular}{|c||c||c||c|}
    \hline
    {\normalsize {\bf Mesh}} & {\normalsize \bf Error in $B^1$} & {\normalsize \bf Error in $B^2$} & {\normalsize \bf Error in $B^3$}    \\ \hline \hline
    {\normalsize $16 \times 32 \times 32$} & {\normalsize $3.074 \times 10^{-4}$}   & {\normalsize $5.469 \times 10^{-4}$} & {\normalsize $5.467 \times 10^{-4}$} \\ \hline
    {\normalsize $32 \times 64 \times 64$} & {\normalsize $1.202 \times 10^{-5}$}   & {\normalsize $2.099 \times 10^{-5}$} & {\normalsize $1.799 \times 10^{-5}$} \\ \hline
    {\normalsize $64 \times 128 \times 128$} & {\normalsize $5.092\times 10^{-7}$}  & {\normalsize $9.562 \times 10^{-7}$} & {\normalsize $8.213 \times 10^{-7}$} \\ \hline \hline
    {\normalsize {\bf EOC}} &  {\normalsize 4.561}	&  	{\normalsize 4.456}	&   {\normalsize 4.453}  \\
    \hline
    \end{tabular}

    \bigskip

    \begin{tabular}{|c||c||c||c|}
    \hline
   {\normalsize {\bf Mesh}}  & {\normalsize \bf Error in $A^1$} & {\normalsize \bf Error in $A^2$} & {\normalsize \bf Error in $A^3$}    \\ \hline \hline
    {\normalsize $16 \times 32 \times 32$} & {\normalsize $3.041\times 10^{-5}$}   & {\normalsize $5.240 \times 10^{-5}$} & {\normalsize $6.238 \times 10^{-5}$} \\ \hline
    {\normalsize $32 \times 64 \times 64$} & {\normalsize $1.280\times 10^{-6}$}   & {\normalsize $2.237 \times 10^{-6}$} & {\normalsize $2.452 \times 10^{-6}$} \\ \hline
    {\normalsize $64 \times 128 \times 128$} & {\normalsize $6.717 \times 10^{-8}$} & {\normalsize $1.111 \times 10^{-7}$} & {\normalsize $1.156 \times 10^{-7}$} \\ \hline \hline
    {\normalsize {\bf EOC}} & {\normalsize 4.252} & {\normalsize 4.331} &  {\normalsize 4.407}   \\
    \hline
    \end{tabular}
    \caption{Convergence study of the 3D Alfv\'en wave for $\text{CFL} = 3.0$. Shown are the $L_{\infty}$-errors at time $t = 1$ of all the magnetic field and magnetic potential components as computed by the WENO-CT3D scheme on various 
    grid resolutions. The magnetic field and magnetic potential values converge at fourth-order accuracy. The experimental order of convergence ({\bf EOC}) is computed by comparing the error for the two finest grids. \label{tab:3dinf}}
\end{Large}
\end{center}
\end{table}

\subsection{2D rotated shock tube problem}
\label{section:2drs}
Next we consider a 1D shock tube problem rotated by an angle of $\alpha$ in a 2D domain. The initial conditions are
\begin{align}
(\rho, u_{\perp},u_{\parallel},u^3,p,B_{\perp},B_{\parallel},B^3) =
\begin{cases}(1,-0.4, 0, 0, 1, 0.75,1,0) & \mbox{if} \quad \xi<0, \\ 
( 0.2, -0.4, 0, 0, 0.1, 0.75,-1,0) & \mbox{if} \quad \xi>0, \end{cases}
\end{align}
where 
\begin{equation}
\xi = x \cos \alpha + y \sin \alpha \quad \text{and} \quad \eta = -x \sin \alpha + y \cos \alpha,
\end{equation}
and $u_{\perp}$ and $B_{\perp}$ are vector components that are perpendicular to the shock interface, and $u_{\parallel}$ and $B_{\parallel}$ are components that are parallel to the shock interface. The initial magnetic potential is
\begin{align}
A^3(0,\xi) =
\begin{cases} 0.75 \, \eta - \xi & \mbox{if} \quad \xi \le 0, \\ 
 0.75 \, \eta + \xi & \mbox{if}  \quad \xi \ge 0. \end{cases}
\end{align}

We solve this problem on the computational domain $(x,y) \in [-1.2,1.2] \times [-1,1]$  with a $180 \times 150$ mesh.
Zero-order extrapolation boundary conditions are used on the left and right boundaries.
On the top and bottom boundaries all the conserved quantities are extrapolated in the direction parallel to the shock interface. 
In addition, to be consistent with zero-order extrapolation boundary condition on $\B$, the linear extrapolation of the magnetic potential $A^3$ is used along the corresponding directions.

Shown in Figure~\ref{2drs} are the density contours of the solutions as computed using the base WENO scheme and the WENO-CT2D scheme. From this figure we note that the solution from the base scheme suffers
from unphysical oscillations that are due to to divergence errors in the magnetic field, while the WENO-CT2D does not
suffer from this problem.
In Figures~\ref{1drs} and~\ref{1drsb} we also present a comparison of the 2D solutions along $y=0$ compared
against a 1D WENO5 solution on a mesh with $N=5000$.  From these figures it is again clear that the base WENO scheme produces extra oscillations and larger errors, while the solution by the WENO-CT2D scheme are in good agreement with the 1D solution.

\subsection{2D Orszag-Tang vortex}
Next we consider the Orszag-Tang vortex problem, which is widely considered as a standard test example for
MHD in the literature (e.g.\ \cite{ article:DaWo98, article:Ro04b, article:To00, article:ZaMaCo94}), since
the solution at late times in the simulation is quite sensitive to divergence errors. 
The problem has a smooth initial condition on the double-periodic box $[0,2 \pi] \times [0, 2\pi]$:
\begin{gather}
\rho(0,x,y)  = \gamma^2, \quad u^1(0,x,y)  = -\sin(y),  \quad u^2 (0,x,y)  = \sin(x),  \\
p(0,x,y)  = \gamma,  \quad B^1(0,x,y)  = -\sin(y),  \quad B^2(0,x,y)  =\sin(2x), \\
 u^3(0,x,y)  = B^3(0,x,y) = 0,
\end{gather}
with the initial magnetic potential:
\begin{align}
A^3(0,x,y) = 0.5 \cos(2x) + \cos(y).
\end{align}
Periodic boundary conditions are imposed on all the boundaries. As time evolves, the solution forms
several shock waves and a vortex structure in the middle of the computational domain.

We solve the MHD equations on a $192 \times 192$ mesh with the WENO-CT2D scheme. In Figure~\ref{OT4}, we present the contours of density at $t=0.5$, 2, 3, and 4. A slice of the pressure at $y =0.625\pi$ and $t=3$ is shown in the right panel of Figure~\ref{pre2d}. Although different papers display the solution at different times and resolutions, our results
are in good agreement with those given in \cite{article:DaWo98, article:Ro04b, article:To00, article:ZaMaCo94}. 
We did not observe significant oscillations in any of the conserved quantities, even when the system is evolved out
to long times.
Our simulation was successfully run to $t=30$ without the introduction of negative pressure anywhere
 in the computational domain. On the other hand, the base scheme without CT step produces negative pressures 
 at around $t=4.0$ on a $192 \times 192$ mesh.



\subsection{Cloud-shock interaction}
Finally we consider the so-called cloud-shock interaction problem, which involves a strong shock interacting with a dense cloud that is in hydrostatic equilibrium. For this problem, we consider the 2D, 2.5D, and 3D versions of the proposed method. 
The 2D and 2.5D versions have the same physical setup. However, 2D means the problem is solved using the WENO-CT2D scheme, and 2.5D means that the magnetic field and the potential are solved using the WENO-CT3D scheme, i.e., with all the components of the magnetic field updated, although all the quantities are still independent of  $z$.

\subsubsection{2D problem}
In this version we consider an MHD shock propagating toward a stationary bubble, with the same setup as the one in \cite{article:Ro04b}. The computation domain is $(x,y) \in [0,1] \times [0, 1]$ with inflow boundary condition at $x=0$ and outflow boundary conditions on the three other sides. The initial conditions consist of a shock initialized at $x=0.05$:
\begin{equation}
\label{2dcs0}
\begin{split}
& \hspace{-3mm} (\rho, u^1, u^2, u^3, p, B^1, B^2, B^3)(0,x,y) \\
& \hspace{-3mm}  = \begin{cases} ( 3.86859, 11.2536, 0, 0, 167.345, 0, 2.1826182, -2.1826182) & \mbox{if} \quad x<0.05, \\ (1, 0, 0, 0, 1, 0, 0.56418958, 0.56418958) & \mbox{if} \quad x > 0.05, \end{cases}
\end{split}
\end{equation}
and a circular cloud of density $\rho = 10$ and radius $r=0.15$ centered at $(x,y) = (0.25,0.5)$. The initial scalar magnetic potential is given by
\begin{align}
A^3 (0,x,y) =  \begin{cases}  -2.1826182(x-0.05) & \mbox{if} \quad x \le 0.05, \\ -0.56418958(x-0.05) & \mbox{if} \quad x\ge 0.05. \end{cases}
\end{align}

The solution is computed on a $256 \times 256$ mesh. Shown in Figure~\ref{2dcs} are schlieren plots of $\ln \rho$ and $\|\B \|$. In general, the solution shows good agreement with the one in \cite{article:Ro04b}, although
the WENO-CT2D schemes shows some higher-resolution features. There is a noticeable extra structure that can be observed  around $x=0.75$ of the density plot (see Figure~\ref{2dcs}(a)). We also find that when the resolution of the solution 
using the $2^{\text{nd}}$-order finite volume solver \cite{article:Ro04b} is doubled to a mesh of $512\times 512$, a similar structure starts to appear\footnote{This test was done using freely available MHDCLAW \cite{mhdclaw} code.}. 
A similar structure can be observed in the solution of Dai and Woodward (Figure 18 in \cite{article:DaWo98}) on a $512 \times 512$ mesh, although their problem setup is slightly different (i.e., they used a stationary shock instead of a stationary bubble). From another perspective, a similar structure around $x=0.75$ can always be observed in the schlieren plots of $\|\B \|$ solved by different methods. Due to those facts, it is reasonable to believe that the solution should consist of this structure and that the high-order solver with less numerical dissipation can obtain this structure with fewer grid points than low-order methods. 

\subsubsection{2.5D problem}
We also consider the 2D version problem with the magnetic field solved by WENO-CT3D scheme so as to compare our 3D and 2D schemes. We call this problem the 2.5D version to be consistent with  \cite{article:HeRoTa10, article:helzel2013}. The problem is initialized in 2.5D with the same initial conditions as the 2D version $\eqref{2dcs0}$. However, as discussed in Section~\ref{sdmpe}, the magnetic potential evolutions of 2D and 3D are significantly different. 

In 2.5D, since all quantities are independent of $z$, the magnetic vector potential satisfies the following evolution equation
\begin{align}
\begin{split}
& A^1_{,t} - u^2 A^2_{,x} - u^3 A^3_{,x}+ u^3 A^1_{,y} = 0, \\
& A^2_{,t} + u^1 A^2_{,x} - u^1A^1_{,y} - u^3 A^3_{,y} = 0, \\
& A^3_{,t} + u^1 A^3_{,x} + u^2 A^3_{,y} = 0.
\end{split}
\end{align}
The magnetic field satisfies
\begin{align}
B^1  = A^3_{,y}, \quad B^2  = -A^3_{,x}, \quad \text{and} \quad B^3  = A^2_{,x} - A^1_{,y}. \label{2.5db3}
\end{align}
 For this version of the cloud-shock problem, the magnetic vector potential is initialized as follows:
\begin{align}
\label{A25d}
\hspace{-4mm}
\Av(0,x,y) = \begin{cases}  (0, -2.1826182(x-0.05), -2.1826182(x-0.05) )^T & \mbox{if} \quad x\le 0.05, \\ (0, 0.56418958(x-0.05), -0.56418958(x-0.05))^T & \mbox{if} \quad x\ge 0.05. \end{cases}
\end{align}

The difference between the 2D and 2.5D schemes is essentially that $B^3$ is not corrected in the CT step in the 2D scheme, while in the 2.5D we update $B^3$ by partial derivatives of $A^1$ and $A^2$ as described above. It is also worthwhile to point out that $B^3_{,z}$ is identically zero in this case, so the update of $B^3$ in $\eqref{2.5db3}$ will not influence the divergence error. In the end, the two approaches produce very similar results.
Shown in Figure~\ref{2dvs25d} are contour plots of $B^3$ using
the two different approaches. For the WENO-CT3D scheme the diffusive limiter described in Section~\ref{3dmp} was used with $\nu = 0.1$.

Although these methods compute $B^3$ in the very different ways, the two solutions in Figure~\ref{2dvs25d} are in good agreement. This result gives us confidence that the proposed 3D scheme is able to solve the problem even with strong shocks. In addition, there are no significant oscillations observed in the 2.5D solution. These results also compare well with
the results of the split finite volume CT approach of \cite{article:HeRoTa10} and by the unsplit MOL finite volume CT approach of \cite{article:helzel2013}.

\subsubsection{3D problem}
The last problem we consider is a fully 3D version of the cloud-shock interaction problem. The initial conditions consist of a shock initialized at $x=0.05$:
\begin{equation}
\begin{split}
& \hspace{-3mm} (\rho, u^1, u^2, u^3, p, B^1, B^2, B^3)(0,x,y,z) = \\
&  \hspace{-3mm} \begin{cases} ( 3.86859, 11.2536, 0, 0, 167.345, 0, 2.1826182, -2.1826182) & \mbox{if} \quad x<0.05, \\ (1, 0, 0, 0, 1, 0, 0.56418958, 0.56418958) & \mbox{if} \quad  x > 0.05, \end{cases}
\end{split}
\end{equation}
and a spherical cloud of density $\rho = 10$ and radius $r=0.15$ centered at $(x,y,z) = (0.25,0.5,0.5)$. The initial conditions for the magnetic potential are the same as $\eqref{A25d}$. The solution is computed on the domain of $[0,1]^3$. Inflow boundary conditions are  imposed at $x=0$ and outflow boundary conditions are used on all other sides. The solution is computed on a $128 \times 128 \times 128$ mesh using the WENO-CT3D scheme. In Figure~\ref{den3d} we show the density of the solution at $t=0.06$, which is in good agreement with the solution in \cite{article:HeRoTa10, article:helzel2013},
although our solution shows less oscillations and higher-resolution compared with previous work. We again observe an extra structure in the density plot, same as in our solution of the 2D problem.  Here the diffusive limiter as described in Section~\ref{3dmp} was used with $\nu = 0.1$.


\section{Conclusions}

In this paper we developed a class of finite difference methods for solving the 2D and 3D ideal MHD equations. 
These  methods are based on high-order finite difference spatial discretizations coupled to high-order
strong stability-preserving Runge--Kutta time-stepping schemes. In order to obtain a discretization of ideal MHD
that exactly respects a discrete form of the divergence-free condition on the magnetic field, a class of novel finite difference schemes based on the WENO approach was developed to solve the evolution equations of the scalar potential in 2D and vector potential in 3D. In particular, an artificial resistivity limiter approach was introduced for 3D problems in order to control the unphysical oscillations in the magnetic field. Overall, the resulting schemes are 4$^\text{th}$-order accurate in space and time for the MHD system.

The numerical methods were tested on several 2D, 2.5D, and 3D test problems, all of which demonstrate
the robustness of our approach. On smooth problems, we achieve fourth-order accuracy in all components, including
the magnetic field and the magnetic potential. For problems with shocks, we are able to accurately capture the shock waves
without introducing unphysical oscillations in any of the solution components. In addition, the cloud-shock interaction problems
also indicated that there is a possible advantage of using a high-order method compared to traditional $2^{\text{nd}}$-order methods. For instance, using a $128 \times 128$ mesh in our methods, we are able to see the same structures that
can only be observed by a 2$^\text{nd}$-order finite volume methods on much finer grid resolutions. This phenomenon is 
observed in both 2D and 3D.
Another advantage of the proposed methods is that they do not involve any multidimensional reconstructions in any step, while traditional high-order finite volume methods commonly need several multidimensional reconstructions in each grid cell. 
For instance,  for the same resolution on a 3D simulation, our finite difference code uses less CPU time than the $3^\text{rd}$-order finite volume code in \cite{article:helzel2013}.

The numerical schemes as developed so far can only be used to solve problems on either a uniform grid or 
on a smoothly varying mapped grid, which is a common disadvantage of high-order ENO/WENO schemes. Thus, our methods are less flexible compared to the finite volume CT methods developed in \cite{article:helzel2013}, in which the scheme has been successfully extended to non-smoothly varying grids. However, a promising approach for overcoming this restriction is to
place the existing WENO-CT method for ideal MHD into an adaptive mesh refinement (AMR) framework. Since our methods
are fully explicit and fully unstaggered, it is possible to incorporate them into the WENO-AMR framework developed by Shen et al. \cite{article:ShQiCh}. We will focus on this generalization in the future work.

\bigskip

\noindent
{\bf Acknowledgements.}
The authors would like to thank Prof. Yingda Cheng for several useful discussions, as
well as the anonymous reviewers for their valuable comments.
AJC was supported by AFOSR grants FA9550-11-1-0281, FA9550-12-1-0343, and FA9550-12-1-0455, NSF grant DMS-1115709, MSU foundation grant SPG-RG100059, and by ORNL under an HPC LDRD. JAR was supported in part by NSF grant DMS-1016202.


\begin{figure}
\begin{center}
\begin{tabular}{cc}
  	(a)\includegraphics[width=0.42\textwidth]{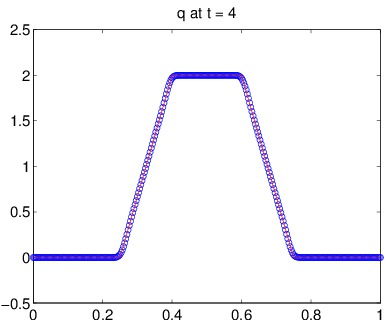} & 
	(b)\includegraphics[width=0.42\textwidth]{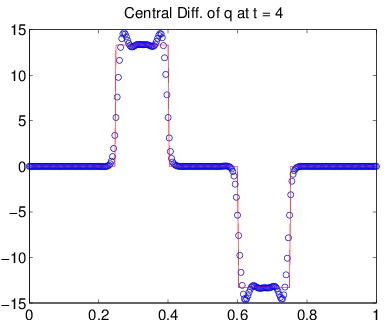} \\
	(c)\includegraphics[width=0.42\textwidth]{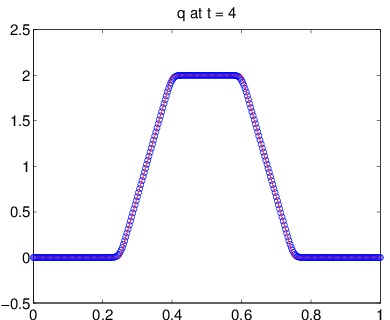} & 
	(d)\includegraphics[width=0.42\textwidth]{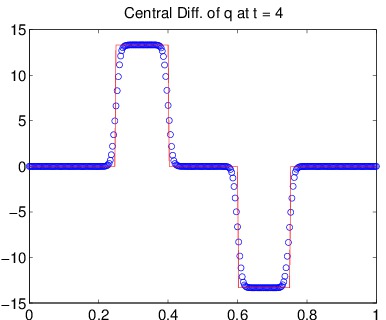}
\end{tabular}
      \caption{The solution to the 1D advection equation on $[0,1]$ with periodic boundary conditions at $t=4$ (i.e., 
      after four full revolutions).  Shown in these panels  are (a)  the solution obtained by the WENO-HCL scheme on a mesh with $N=300$,  (b) the derivative of this solution as computed with a
      $4^{\text{th}}$-order central difference approximation,  
      (c)  the solution obtained by the WENO-HJ scheme on a mesh with $N=300$,  and
      (d) the derivative of this solution as computed with a
      $4^{\text{th}}$-order central difference approximation.
      The proposed WENO-HJ scheme
      is able to control unphysical oscillations both in the solution and its derivative.      
      \label{1dadcp}}
\end{center}
\end{figure}

\begin{figure}
\begin{center}
\begin{tabular}{cc}
(a)\includegraphics[width=0.42\textwidth]{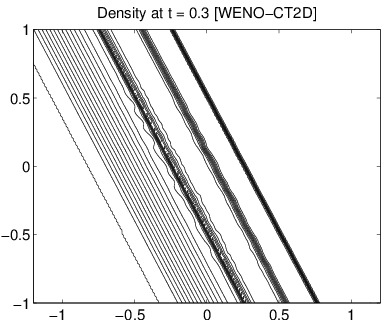} &
(b)\includegraphics[width=0.42\textwidth]{den2d}
\end{tabular}
      \caption{The rotated shock tube problem. Shown in these panels are density contours of (a) the base WENO scheme (i.e., no constrained transport) and (b) the WENO-CT2D scheme. In this case, $\alpha = \tan^{-1} (0.5)$. 30 equally spaced contours are shown for each graph in the ranges $\rho \in [0.1795, 1]$. \label{2drs}}
\end{center}
\end{figure}

\begin{figure}
\begin{center}
\begin{tabular}{cc}
(a)\includegraphics[width=0.42\textwidth]{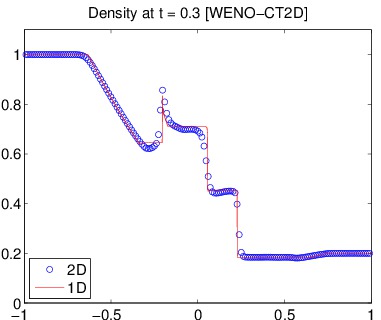} &
(b)\includegraphics[width=0.42\textwidth]{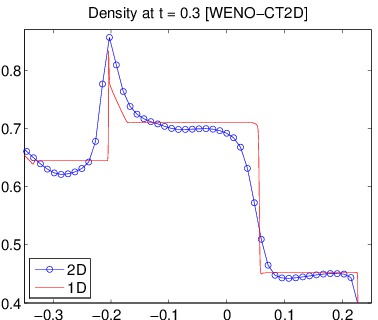} \\
(c)\includegraphics[width=0.42\textwidth]{den1d} &
(d)\includegraphics[width=0.42\textwidth]{denzm}
\end{tabular}
      \caption{Cut along $y=0$ of the rotated shock tube problem. Shown in these panels are (a) a 1D cut along $y=0$
      of the density as computed with the base WENO scheme
      (i.e., without constrained transport), (b) a zoomed in view of the same plot, (c) a 1D cut along $y=0$
      of the density as computed with the WENO-CT2D scheme, and (d) a zoomed in view of the same plot. The solid line in each plot is a highly resolved solution of the 1D shock tube problem by the 1D WENO5 scheme on a uniform mesh with $N=5000$.  \label{1drs}}
\end{center}
\end{figure}

\begin{figure}
\begin{center}
\begin{tabular}{cc}
(a)\includegraphics[width=0.42\textwidth]{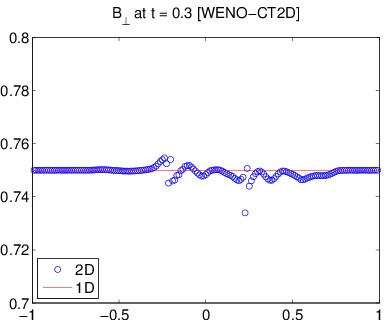} &
(b)\includegraphics[width=0.42\textwidth]{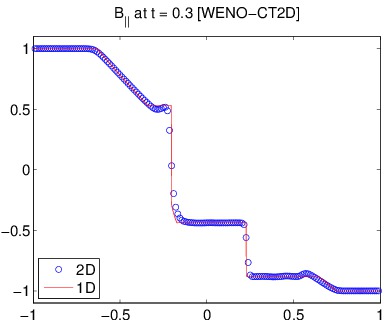} \\
(c)\includegraphics[width=0.42\textwidth]{bper} &
(d)\includegraphics[width=0.42\textwidth]{bprl}
\end{tabular}
      \caption{Cut along $y=0$ of the rotated shock tube problem. Shown in these panels are  1D cuts along $y=0$
      of the magnetic field (a) perpendicular ($B_\perp$) and (b) parallel ($B_\parallel$) to the shock interface as
      computed with the base WENO scheme  (i.e., without constrained transport) and
      1D cuts along $y=0$
      of the magnetic field (c) perpendicular ($B_\perp$) and (d) parallel ($B_\parallel$) to the shock interface as
      computed with the WENO-CT2D dscheme.
             The solid line in each plot is a highly resolved solution of the 1D shock tube problem by the 1D WENO5 scheme on a uniform mesh with $N=5000$. \label{1drsb}}
\end{center}
\end{figure}

\begin{figure}
\begin{center}
\begin{tabular}{cc}
(a)\includegraphics[width=0.42\textwidth]{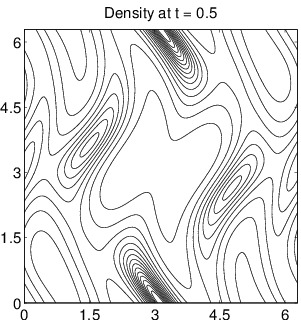} &
(b)\includegraphics[width=0.42\textwidth]{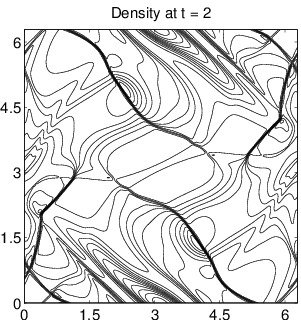} \\
(c)\includegraphics[width=0.42\textwidth]{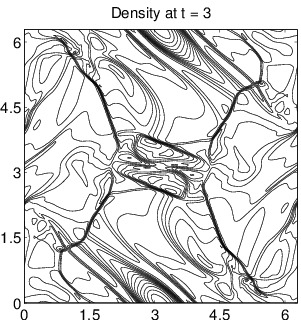} &
(d)\includegraphics[width=0.42\textwidth]{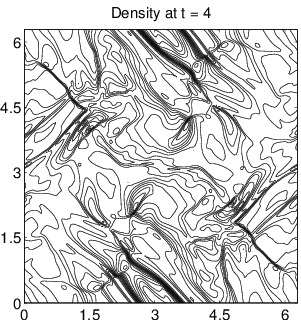}
\end{tabular}
      \caption{The Orszag-Tang vortex problem. Shown in these panels are the density at times (a) $t=0.5$, (b) $t=2$, (c) $t=3$, and (d)
      $t=4$ as computed with the WENO-CT2D scheme on a $192 \times 192$ mesh. 
      15 equally spaced contours were used for each plot. \label{OT4}}
\end{center}
\end{figure}

\begin{figure}
\begin{center}
\begin{tabular}{cc}
(a)\includegraphics[width=0.42\textwidth]{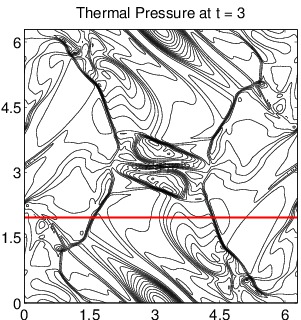} &
(b)\includegraphics[width=0.42\textwidth]{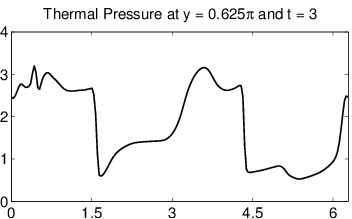}
\end{tabular}
      \caption{The Orszag-Tang vortex problem. Shown in these panels are (a) the thermal pressure
       as computed with the WENO-CT2D scheme at $t = 3$ on a $192 \times 192$ mesh and (b)
      a slice of the thermal pressure at $y = 0.625 \pi$. \label{pre2d}}
\end{center}
\end{figure}

\begin{figure}
\begin{center}
\begin{tabular}{cc}
(a)\includegraphics[width=0.42\textwidth]{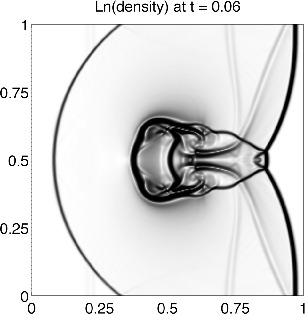} &
(b)\includegraphics[width=0.42\textwidth]{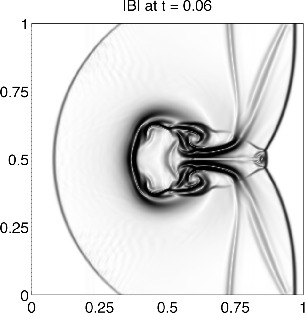}
\end{tabular}
      \caption{The 2D cloud-shock interaction problem. Shown in these panels are schlieren plots at time $t=0.06$ of (a)
      the natural log of the density  and (b) the norm of magnetic field. The solution was obtained using the
       WENO-CT2D scheme on a $256 \times 256$ mesh.  \label{2dcs}}
\end{center}
\end{figure}

\begin{figure}
\begin{center}
\begin{tabular}{cc}
(a)\includegraphics[width=0.42\textwidth]{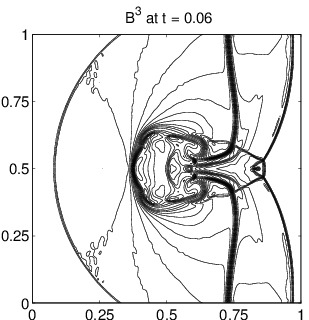} &
(b)\includegraphics[width=0.42\textwidth]{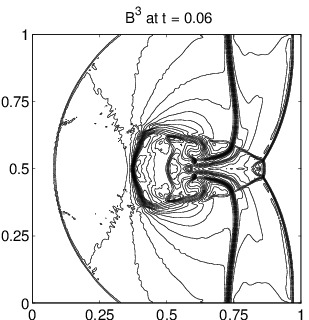}
\end{tabular}
      \caption{The 2D cloud-shock interaction problem. Shown in these panels are the $B^3$ component of the magnetic field at time $t = 0.06$ on a $256 \times 256$ mesh as solved with (a) the WENO-CT2D scheme and (b) the 2.5D version of the WENO-CT3D scheme.  Here the 2.5D problem was implemented in such a way that all three components of the magnetic field were updated in the constrained transport step, while in the 2D problem only first two components of the magnetic field were updated. 25 equally spaced contours were used in each of these panels. \label{2dvs25d}}
\end{center}
\end{figure}

\begin{figure}
  \begin{center}
    \includegraphics[width=1\textwidth]{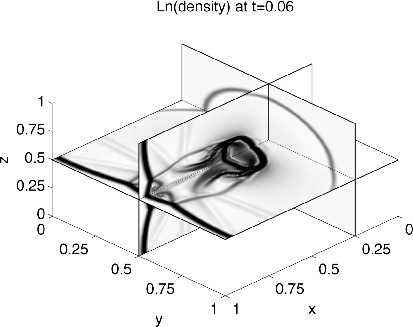}
      \caption{The 3D cloud-shock interaction problem. Shown in this plot is the natural log of the density as
      computed with the WENO-CT3D scheme on a $128 \times 128 \times 128$ mesh. \label{den3d}}
    \end{center}
\end{figure}

\appendix

\section{WENO reconstruction}
\label{sec:appendix}

The main idea of WENO reconstruction is to compute a finite difference stencil using a weighted average of several
smaller stencils. The weights are chosen based on the smoothness of the solution on each of the smaller stencils.
The full procedure can be found in  \cite{article:JiShu96, article:SeYaCh13, book:Sh98, article:Sh09}.
For completeness, we also include a brief description of the $5^{\text{th}}$-order WENO reconstruction as used
in this work and define the operator $\Phi_{\text{WENO}5}$ that was used in Section \ref{sdmhd}.

We consider the problem on a uniform grid with $N+1$ grid points,
\begin{align}
a = x_{0} < x_{1} < \dots < x_{N} = b,
\end{align}
and let the cell averages of some function $h(x)$ on the interval $I_i = \left(x_{i-\half},x_{i+\half} \right)$
be denoted by
\begin{align}
\bar{h}_i = \frac{1}{\Delta x} \int_{x_{i-\half}}^{x_{i+\half}} h(x) \, dx.
\end{align}
We would like to approximate the value of $h(x)$ at the half node $x_{i+\half}$ by WENO reconstruction on the stencil:
 $S = \{I_{i-2}, I_{i-1}, \ldots, I_{i+2}\}$. 
There are three sub-stencils for node $x_{i+\half}$:  $S_0 = \{I_{i-2}, I_{i-1}, I_{i} \}$, $S_1 = \{I_{i-1}, I_{i}, I_{i+1} \}$ and $S_2 = \{I_{i}, I_{i+1}, I_{i+2} \}$. 
Through a simple Taylor expansions of
$h(x)$, we can obtain $3^{\text{rd}}$-order accurate approximations of $h^{(i)}_{i+\half}$ on each sub-stencil $S_i$ as follows:
\begin{align}
h^{(0)}_{i+\half} & = \frac{1}{3} \bar{h}_{i-2} - \frac{7}{6}\bar{h}_{i-1} + \frac{11}{6} \bar{h}_{i}, \\
h^{(1)}_{i+\half} & = -\frac{1}{6} \bar{h}_{i-1} + \frac{5}{6}\bar{h}_{i} + \frac{1}{3} \bar{h}_{i+1}, \\
h^{(2)}_{i+\half} & = \frac{1}{3} \bar{h}_{i} + \frac{5}{6}\bar{h}_{i+1} - \frac{1}{6} \bar{h}_{i+2}.
\end{align}
The approximation $h_{i+\half}$ is then defined as a linear convex combination of the above three approximations:
\begin{align}
h_{i+\half} = w_0 \, h^{(0)}_{i+\half} + w_1 \, h^{(1)}_{i+\half} + w_2 \, h^{(2)}_{i+\half},
\end{align}
where the nonlinear weights are defined as 
\begin{gather}
w_j = \frac{{\tilde{w}_j}}{{\tilde{w}_0}+{\tilde{w}_1}+{\tilde{w}_2}}, \\
{\tilde{w}_0} = \frac{1}{(\epsilon+\beta_0)^2}, \quad
{\tilde{w}_1} = \frac{6}{(\epsilon+\beta_1)^2}, \quad
{\tilde{w}_2} = \frac{3}{(\epsilon+\beta_2)^2}.
\end{gather}
In our computations we take $\epsilon = 10^{-6}$ and the smoothness indicator parameters, $\beta_i$, are chosen as in \cite{article:JiShu96}:
\begin{align}
\beta_0 & = \frac{13}{12}(\bar{h}_{i-2}-2\bar{h}_{i-1}+\bar{h}_i)^2+\frac{1}{4}(\bar{h}_{i-2}-4\bar{h}_{i-1}+3\bar{h}_i)^2, \\
\beta_1 & = \frac{13}{12}(\bar{h}_{i-1}-2\bar{h}_{i}+\bar{h}_{i+1})^2+\frac{1}{4}(\bar{h}_{i-1}-\bar{h}_{i+1})^2, \\
\beta_2 & = \frac{13}{12}(\bar{h}_{i}-2\bar{h}_{i+1}+\bar{h}_{i+2})^2+\frac{1}{4}(3\bar{h}_{i}-4\bar{h}_{i+1}+\bar{h}_{i+2})^2.
\end{align}
From these  we define $\Phi_{\text{WENO}5}$ in Section \ref{sdmhd} as follows:
\begin{align}
\Phi_{\text{WENO}5} (\bar{h}_{i-2},\bar{h}_{i-1},\bar{h}_{i},\bar{h}_{i+1},\bar{h}_{i+2}) : = w_0 \, h^{(0)}_{i+\half} + w_1 \, h^{(1)}_{i+\half} + w_2 \, h^{(2)}_{i+\half}.
\end{align}
The approximation value $h_{i+\half}$ has the following properties: 
\begin{enumerate}
\item If $h(x)$ is smooth in the full stencil $S$, $h_{i+\half}$ is a $5^{\text{th}}$-order accurate approximation to the value $h\left(x_{i+\half} \right)$. \\
\item If $h(x)$ is not smooth or has discontinuity in the full stencil $S$, the nonlinear weights are computed in such a way that $h_{i+\half}$ is mainly reconstructed from the locally smoothest sub-stencil. Consequently, the spurious oscillations can be effectively controlled.
\end{enumerate}

\bibliographystyle{plain}

\end{document}